\documentclass[english,reqno]{amsart}

\usepackage{enumitem} 
\usepackage{mathtools} 
\usepackage[table]{xcolor} 
\usepackage[all]{xy} 
\usepackage{tikz} 
\usepackage{babel} 
\usepackage{setspace} 

\usepackage[colorlinks,linkcolor=red,anchorcolor=green,citecolor=blue]{hyperref} 
\hypersetup{linktocpage = true} 

\usepackage{rotating} 
\usepackage{ytableau} 

\usepackage{mathpazo}
\usepackage{mathrsfs}
\DeclareFontFamily{OMS}{rsfs}{\skewchar\font'60}
\DeclareFontShape{OMS}{rsfs}{m}{n}{<-5>rsfs5 <5-7>rsfs7 <7->rsfs10 }{}
\DeclareSymbolFont{rsfs}{OMS}{rsfs}{m}{n}
\DeclareSymbolFontAlphabet{\scr}{rsfs}
\DeclareSymbolFontAlphabet{\scr}{rsfs}

\newcommand\cE{{\mathcal E}}
\newcommand\cF{{\mathcal F}}
\newcommand\cG{{\mathcal G}}

\newcommand\cO{{\mathcal O}}

\newcommand\cQ{{\mathcal Q}}

\newcommand\bbC{{\mathbb C}}

\newcommand\bbN{{\mathbb N}}
\newcommand\bbP{{\mathbb P}}

\newcommand\bbZ{{\mathbb Z}}

\newcommand{\rk}{{\rm rk}}
\newcommand{\codim}{{\rm codim}}

\newcommand{\Bs}{{\rm Bs}}

\DeclareMathOperator*{\bs}{Bs}

\DeclareMathOperator*{\pic}{Pic}

\newcommand{\cHom}[2]{\ensuremath{\mathcal{H}om_{\mathcal{O}_X}(#1,#2)}}

\newtheorem{lemma1}{}[section]
\newenvironment{lemma}{\begin{lemma1}{\bf Lemma.}}{\end{lemma1}}
\newenvironment{example}{\begin{lemma1}{\bf Example.}\rm}{\end{lemma1}}
\newenvironment{thm}{\begin{lemma1}{\bf Theorem.}}{\end{lemma1}}
\newenvironment{prop}{\begin{lemma1}{\bf Proposition.}}{\end{lemma1}}
\newenvironment{cor}{\begin{lemma1}{\bf Corollary.}}{\end{lemma1}}
\newenvironment{remark}{\begin{lemma1}{\bf Remark.}\rm}{\end{lemma1}}

\newenvironment{defn}{\begin{lemma1}{\bf Definition.}}{\end{lemma1}}

\newenvironment{thm A}{{\bf Theorem A.}}{}
\newenvironment{thm B}{{\bf Theorem B.}}{}
\newenvironment{thm C}{{\bf Theorem C.}}{}
\newenvironment{thm D}{{\bf Theorem D.}}{}
\newenvironment{remark*}{{\bf Remark.}}{}
\newenvironment{example*}{{\bf Example.}}{}
\newenvironment{assumption*}{{\bf Assumption.}}{}

\setlist[itemize]{leftmargin=*}
\setlist[enumerate]{leftmargin=*}

\numberwithin{equation}{section} 

\makeatletter

\ifnum\@ptsize=0  \fi \ifnum\@ptsize=2  \fi

\title{Stability of tangent bundles of complete intersections and effective restriction} 
\date{\today}


\author{Jie Liu}

\address{Jie Liu, Morningside Center of Mathematics, Academy of Mathematics and Systems Science, Chinese Academy of Sciences, Beijing, 100190, China}
\email{jliu@amss.ac.cn}

\begin{document}

\begin{abstract}
	For $n\geq 3$, let $M$ be an $(n+r)$-dimensional irreducible Hermitian symmetric space of compact type and let $\cO_M(1)$ be the ample generator of $\pic(M)$. Let $Y=H_1\cap\dots\cap H_r$ be a smooth complete intersection of dimension $n$ where $H_i\in\vert \cO_M(d_i)\vert$ with $d_i\geq 2$. We prove a vanishing theorem for twisted holomorphic forms on $Y$. As an application, we show that the tangent bundle $T_Y$ of $Y$ is stable. Moreover, if $X$ is a smooth hypersurface of degree $d$ in $Y$ such that the restriction $\pic(Y)\rightarrow \pic(X)$ is surjective, we establish some effective results for $d$ to guarantee the stability of the restriction $T_Y\vert_X$. In particular, if $Y$ is a general hypersurface in $\bbP^{n+1}$ and $X$ is general smooth divisor in $Y$, we show that $T_Y\vert_X$ is stable except for some well-known examples. We also address the cases where the Picard group increases by restriction.
\end{abstract}

\maketitle

\tableofcontents

\vspace{-0.2cm}

\section{Introduction}

It has been one of main problems in K{\"a}hler geometry to study  which Fano manifolds with $b_2=1$ admit a K{\"a}hler-Einstein metric. The celebrated Yau-Tian-Donaldson conjecture asserts that a Fano manifold admits a K{\"a}hler-Einstein metric if and only if it is $K$-polystable. This conjecture has been solved recently (see \cite{ChenDonaldsonSun2014,ChenDonaldsonSun2015,Tian2015} and the references therein). A weaker and more algebraic question related to the existence of K{\"a}hler-Einstein metrics asks whether the tangent bundle is (semi-)stable with respect to the anti-canonical divisor. Let us recall the definition of stability. Let $(Z,H)$ be an $n$-dimensional polarized projective manifold, and let $\cE$ be a non-zero torsion-free coherent sheaf over $Z$. The \emph{slope} of $\cE$ with respect to $H$ is defined to be
\[\mu(\cE)\colon=\frac{c_1(\cE)\cdot H^{n-1}}{\rk(\cE)}.\] 

\begin{defn}
	Let $(Z,H)$ be a polarized projective manifold, and let $\cE$ be a non-zero torsion-free coherent sheaf over $Z$. The sheaf $\cE$ is called $H$-stable (resp. $H$-semi-stable) if for any non-zero coherent subsheaf $\cF\subset \cE$ with $0<\rk(\cF)<\rk(\cE)$, we have
	\[\mu(\cF)<\mu(\cE)\qquad (resp.\quad \mu(\cF)\leq \mu(\cE)).\]
\end{defn}

By the works of Ramanan, Umemura, Azad-Biswas, Reid, Peternell-Wi{\'s}niewski and Hwang, the stability of tangent bundles is known for homogeneous spaces \cite{Ramanan1966,Umemura1978,AzadBiswas2010}, Fano manifolds with index one \cite{Reid1977}, Fano manifolds of dimension at most six \cite{PeternellWisniewski1995,Hwang1998}, general complete intersections in $\bbP^N$ \cite{PeternellWisniewski1995} and prime Fano manifolds with large index \cite{Hwang2001}. If $X$ is a projective manifold with $\pic(X)\cong\bbZ \cO_X(1)$, then the stability of $T_X$ is very much related to the cohomology vanishings of type
\[H^q(X,\Omega_X^p(\ell)),\]
where $\Omega_X^p(\ell)=\Omega_X^p\otimes\cO_X(\ell)$. If $M$ is an irreducible Hermitian symmetric space of compact type, then $M$ is a Fano manifold with Picard number one. In particular, making use of the work of Kostant \cite{Kostant1961}, Snow developed an algorithm in \cite{Snow1986} and \cite{Snow1988} to determine whether a given cohomology group $H^q(M,\Omega_M^p(\ell))$ vanishes. In particular, Biswas, Chaput and Mourougane proved in \cite{BiswasChaputMourougane2015} the following vanishing theorem by using the algorithm of Snow.

\begin{thm}\cite[Theorem D]{BiswasChaputMourougane2015}
	Let $M$ be an $n$-dimensional irreducible Hermitian symmetric space of compact type which is not isomorphic to a projective space. Let $\cO_M(1)$ be the ample generator of $\pic(M)$, and let $\ell$ and $p$ be two positive integers such that $H^q(M,\Omega_M^p(\ell))\not=0$ for some $q\geq 0$. Then we have
	\[\ell+q\geq p\frac{r_M}{n},\]
	where $r_M$ is the index of $M$, i.e., $\cO_M(-K_M)\cong\cO_M(r_M)$.
\end{thm}

 As an application of this vanishing theorem, we prove the following similar vanishing theorem for smooth complete intersections in $M$.
 
 \begin{thm}\label{Vanishing complete intersection}
 	For $n\geq 2$ and $r\geq 1$, let $M$ be an $(n+r)$-dimensional irreducible Hermitian symmetric space of compact type, and let $\cO_M(1)$ be the ample generator of $\pic(M)$. Let $Y=H_1\cap\dots\cap H_r$ be a smooth complete intersection of dimension $n$ where $H_i\in\vert \cO_M(d_i)\vert$ with $d_i\geq 2$. Denote $\cO_M(1)\vert_Y$ by $\cO_Y(1)$. Then for any $q\geq 0$, $p\geq 1$, $\ell\in\bbZ$ such that $q+p\leq n-1$, the following statements hold.
 	\begin{enumerate}
 		\item If $H^q(Y,\Omega_{Y}^p)\not=0$, then $p=q$.
 		
 		\item If $H^q(Y,\Omega_{Y}^p(\ell))\not=0$ for some $\ell\not=0$, then 
 		\[\ell+q>p\frac{r_Y}{n},\]
 		where $r_Y=r_M-d_1-\dots-d_r$.
 	\end{enumerate}
 \end{thm}
 
 In Theorem \ref{Vanishing complete intersection}, if $n\geq 3$, the natural map $\pic(M)\rightarrow\pic(Y)$ is an isomorphism by Lefschetz hyperplane theorem \cite[Corollary 3.3]{Hartshorne1970}. In particular, we have $\pic(Y)\cong \bbZ\cO_Y(1)=\bbZ\cO_M(1)\vert_Y$. As the first application of our vanishing theorem, we generalize a result of Peternell-Wi{\'s}niewski (see \cite[Theorem 1.5]{PeternellWisniewski1995}) in the following theorem.

\begin{thm}\label{stability hypersurfaces of Hermitian symmetric spaces}
	For $n\geq 3$ and $r\geq 1$, let $M$ be an $(n+r)$-dimensional irreducible Hermitian symmetric space of compact type, and let $\cO_M(1)$ be the ample generator of $\pic(M)$. If $Y=H_1\cap\dots\cap H_r$ is a smooth complete intersection of dimension $n$ where $H_i\in\vert\cO_M(d_i)\vert$ with $d_i\geq 2$, then the tangent bundle $T_Y$ is stable.
\end{thm}

Let $(Z,H)$ be a polarized projective manifold, and let $Y\in \vert dH\vert$ be a general smooth hypersurface of degree $d$. Let $\cE$ be a torsion-free coherent sheaf over $Z$. Then it is easy to see that, if $\cE$ is an $H$-unstable sheaf, then $\cE\vert_Y$ is $H\vert_Y$-unstable. Equivalently, $\cE$ is semi-stable if $\cE\vert_Y$ is $H\vert_Y$-semi-stable. Nevertheless, the converse is false in general as shown by the following easy counterexample.

\begin{example}
	The tangent bundle $T_{\bbP^n}$ of $\bbP^n$ is stable with $\mu(T_{\bbP^{n}})=(n+1)/n$. However, if $Y$ is a hyperplane, then the restriction $T_{\bbP^n}\vert_Y$ is unstable since $T_Y$ is a subbundle of $T_{\bbP^n}\vert_Y$ with $\mu(T_Y)=n/(n-1)$.
\end{example}

However, by a result of Mehta-Ramanathan, the restriction of a (semi-)stable sheaf is (semi-)stable for sufficiently large $d$. In general, we have the following important effective restriction theorem (cf. \cite{Flenner1984,MehtaRamanathan1984,Langer2004}).

\begin{thm}\cite[Theorem 5.2 and Corollary 5.4]{Langer2004}
	Let $(Z,H)$ be a polarized projective manifold of dimension $n$. Let $\cE$ be a torsion-free $H$-(semi-)stable sheaf of rank $p\geq 2$. Let $Y\in\vert dH\vert$ be a general smooth hypersurface. If
	\[d>\frac{p-1}{p}\Delta(\cE)H^{n-2}+\frac{1}{p(p-1)H^n},\]
	then $\cE\vert_Y$ is $H\vert_Y$-(semi-)stable. Here $\Delta(\cE)=2p c_2(\cE)-(p-1)c^2_1(\cE)$ is the discriminant of $\cE$.
\end{thm}

In \cite{BiswasChaputMourougane2015}, an optimal effective theorem is established for tangent bundles of irreducible Hermitian symmetric spaces of compact type. 

\begin{thm}\cite[Theorem A and B]{BiswasChaputMourougane2015}\label{BCM Theorem A}
	For $n\geq 3$, let $M$ be an $n$-dimensional irreducible Hermitian symmetric space of compact type, and let $Y$ be a smooth hypersurface of $M$. Then the restriction $T_M\vert_Y$ is stable unless $Y$ is a linear section and $M$ is isomorphic to either $\bbP^n$ or a smooth quadric hypersurface $Q^n$.
\end{thm}

In \cite{BiswasChaputMourougane2015}, this theorem was stated for the cotangent bundle $\Omega_M^1$ of $M$, which is equivalent to Theorem \ref{BCM Theorem A}, since a vector bundle $E$ is $H$-stable over a polarized projective manifold $(Z,H)$ if and only if its dual bundle $E^*$ is $H$-stable. 

As the second application of Theorem \ref{Vanishing complete intersection}, we reduce the effective restriction problem of tangent bundles to the existence of certain twisted vector fields (cf. Proposition \ref{reduce to exitence of vector fields}), and then we can derive the following result.

\begin{thm}\label{Effective restriction HSS}
	For $n\geq 3$, let $M$ be an $(n+r)$-dimensional irreducible Hermitian symmetric space of compact type, and let $\cO_M(1)$ be the ample generator of $\pic(M)$. Let $Y=H_1\cap\dots\cap H_r$ be a smooth complete intersection of dimension $n$ where $H_i\in\vert\cO_M(d_i)\vert$ with $2\leq d_1\leq\dots\leq d_r$. Let $X\in\vert\cO_Y(d)\vert$ be a smooth hypersurface of degree $d$. Assume moreover that the composite of restrictions
	\[\pic(M)\rightarrow\pic(Y)\rightarrow\pic(X)\]
	is surjective. Then the restriction $T_Y\vert_X$ is stable if one of the following conditions holds.
	\begin{enumerate}
		\item $Y$ is a Fano manifold and $M$ is not isomorphic to the projective space $\bbP^{n+r}$ or any smooth quadric hypersurface $Q^{n+r}$.
		
		\item $Y$ is a Fano manifold, $d\geq d_1$ and $M$ is isomorphic to the projective space $\bbP^{n+r}$.
		
		\item $Y$ is a Fano manifold, $d\geq 2$ and $M$ is isomorphic to a smooth quadric hypersurface $Q^{n+r}$.
		
		\item $X$ is general and $d>d_r-r_Y/n$, where $r_Y=r_M-d_1-\dots-d_r$.
	\end{enumerate}
\end{thm}

In the case where $Y$ is a general smooth hypersurface in $\bbP^{n+1}$, using the strong Lefschetz property of the Milnor algebra of $Y$, we can prove an extension theorem for twisted vector fields on $X$ (see Theorem \ref{Globalsections}), and an optimal answer to the effective restriction problem can be given in this setting.

\begin{thm}\label{Restrition Invariant Picard Group}
	For $n\geq 3$, let $Y$ be a general smooth hypersurface in the projective space $\bbP^{n+1}$, and let $X\in \vert\cO_Y(d)\vert$ be a general smooth hypersurface of degree $d$ in $Y$. Assume furthermore that the restriction homomorphism $\pic(Y)\rightarrow\pic(X)$ is surjective, then $T_Y\vert_X$ is stable unless $d=1$, and $Y$ is isomorphic to either $\bbP^n$ or $Q^n$.
\end{thm}

In each exceptional case, the tangent bundle of $X$ will destabilize $T_Y\vert_X$, so our result above is sharp. The stability of restrictions of tangent bundles with an increase of Picard group was also considered in \cite{BiswasChaputMourougane2015}. According to Lefschetz's hyperplane theorem, the map $\pic(Y)\rightarrow\pic(X)$ is always surjective if $n\geq 4$. In fact, Lefschetz proved an even more general version, the so-called Noether-Lefschetz theorem, in \cite{Lefschetz1921}: a very general complete intersection surface $X$ in $\bbP^N$ contains only curves that are themselves complete intersections unless $X$ is an intersection of two quadric threefolds in $\bbP^4$, or a quadric surface in $\bbP^3$, or a cubic surface in $\bbP^3$ (see also \cite{Green1988,Kim1991}). In these exceptional cases, the possibilities of the pair $(Y,X)$ are as follows:
\begin{enumerate}[label=(\arabic*)]
	\item The manifold $Y$ is the projective space $\bbP^3$ and $X$ is a quadric surface or a cubic surface.
	
	\item The manifold $Y\subset\bbP^4$ is a quadric threefold and $X$ is a linear section or a quadric section.
	
	\item The manifold $Y\subset\bbP^4$ is a cubic threefold and $X$ is a linear section of $Y$.
\end{enumerate}

When $Y$ is a quadric threefold or the projective space $\bbP^3$, according to \cite[Theorem B]{BiswasChaputMourougane2015}, the restriction $T_Y\vert_X$ is semi-stable with respect to $-K_X$ unless $Y$ and $X$ are both projective spaces; and $T_Y\vert_X$ is stable with respect to $-K_X$ if $X$ is not a linear section. In the following result, we address the stability of the restriction $T_Y\vert_X$ in the case where $Y$ is a cubic threefold and $X$ is a linear section.

\begin{thm}\label{Semi-stability Cubic}
	Let $Y\subset\bbP^4$ be a general cubic threefold, and let $X\in\vert\cO_Y(1)\vert$ be a general smooth linear section. Then the restriction $T_Y\vert_X$ is stable with respect to $\cO_X(1)$.
\end{thm}

This paper is organized as follows. In Section \ref{Hermitian Lefshetz}, we introduce the basic notions concerning Hermitian symmetric spaces and the Lefschetz properties of Artinian algebras. In Section \ref{Section Vanishing theorems}, we collect the results about the cohomologies of $(n-1)$-forms of Hermitian symmetric spaces and we introduce the notion of special cohomologies. In Section \ref{Section Extension of vector fields}, we investigate the twisted vector fields over complete intersections in Hermitian symmetric spaces and we prove some extension results in various settings. In Section \ref{Section Stability Invariant Picard group}, we address the stability of tangent bundles of complete intersections in Hermitian symmetric spaces and study the effective restriction problem. In particular, we prove Theorem \ref{Vanishing complete intersection}, Theorem \ref{stability hypersurfaces of Hermitian symmetric spaces}, Theorem \ref{Effective restriction HSS} and Theorem \ref{Restrition Invariant Picard Group}. In Section \ref{Section Hyperplane of cubic threefolds}, we consider the case where the Picard number increases by restriction and prove Theorem \ref{Semi-stability Cubic}. 

\subsection*{Convention} For an $n$-dimensional projective variety $Y$, we denote by $\Omega_Y^1$ the sheaf of K\"ahler differentials and denote by $\Omega^p_Y$ the sheaf $\wedge^p\Omega_Y^1$. For a line bundle $\cO_Y(1)$ and a coherent sheaf $\cF$ on $Y$, we will denote $\cF\otimes\cO_Y(\ell)$ by $\cF(\ell)$, and the number $h^i(Y,\cF)$ is the dimension of $H^i(Y,\cF)$ over $\bbC$. Moreover, the dual sheaf $\cHom{\cF}{\cO_Y}$ of $\cF$ is denoted by $\cF^*$. If $Y$ is smooth, the canonical divisor, denoted by $K_Y$, is a Weil divisor associated to $\Omega_Y^n$. For a submanifold $X$ of a polarized manifold $(Y,\cO_Y(1))$, we denote by $\cO_X(1)$ the restriction $\cO_Y(1)\vert_X$. If $Y$ is a projective manifold with Picard number one, we consider always the polarization given by the ample generator $\cO_Y(1)$ of $\pic(Y)$ unless otherwise stated.

\subsection*{Acknowledgements} I would like to thank my thesis advisor Christophe Mourougane for bringing my attention to this problem, and also for his helpful discussions, suggestions and constant encouragements. It is a pleasure to thank Pierre-Emmanuel Chaput, St\'ephane Druel, Wei-Guo Foo, Andreas H\"oring and Zhenjian Wang for useful discussion and helpful comments. This paper was written during my stay at Institut de Recherche Mathématique de Rennes (IRMAR) and Laboratoire de Math\'ematiques J.~A. Dieudonn\'e (LJAD), and I would like to thank both institutions for their hospitality and their support. I want to thank the referee for suggesting the spectral sequence argument to simplify the proofs of Theorem \ref{Vanishing complete intersection} and Proposition \ref{Extension of vector field}, and to drop the genericity assumption in the previous version of this paper.

\section{Hermitian symmetric spaces and Lefschetz properties}\label{Hermitian Lefshetz}

In this section, we collect some basic materials about Hermitian symmetric spaces of compact type and the Lefschetz properties of Artinian algebras. We refer to \cite{BorelHirzebruch1958} and \cite{MiglioreNagel2013} for further details. 

\subsection{Hermitian symmetric spaces}

Let $(M,g)$ be a Riemannian manifold. A non-trivial isometry $\sigma$ of $(M,g)$ is said to be an \emph{involution} if $\sigma^2=id$. A Riemannian manifold $(M,g)$ is said to be \emph{Riemannian symmetric} if at each point $x\in M$ there exists an involution $\sigma_x$ such that $x$ is an isolated fixed point of $\sigma_x$.

\begin{defn}
	Let $(M,g)$ be a Riemannian symmetric manifold. Then $(M,g)$ is said to be a Hermitian symmetric manifold if $(M,g)$ is a Hermitian manifold and the involution $\sigma_x$ at each point $x\in M$ can be chosen to be a holomorphic isometry.
\end{defn}

A Hermitian symmetric space $M$ is called \emph{irreducible} if it cannot be written as the non-trivial product of two Hermitian symmetric spaces. It is well-known that the irreducible Hermitian symmetric spaces of compact type are Fano manifolds of Picard number one. We will denote by $\cO_M(1)$ the ample generator of $\pic(M)$. In this case, the index of $M$ is defined to the unique positive integer $r_M$ such that $\cO_M(-K_M)\cong\cO_M(r_M)$.

The Hermitian symmetric spaces are homogeneous under their isometry groups. According to Cartan, there are exactly six types of irreducible Hermitian symmetric spaces of compact type: Grassmannians (type $A_n$), quadric hypersurfaces (type $B_n$ or $D_n$), Lagrangian Grassmannians (type $C_n$), spinor Grassmannians (type $D_n$) and two exceptional cases (type $E_6$ and $E_7$).

\subsection{Lefschetz properties of Artinian algebras} 

Let $R=\bbC[x_1,\dots,x_r]$ be the graded polynomial ring in $r$ variables over $\bbC$. Let
\[A=R/I=\bigoplus_{i=0}^n A_i\]
be a graded Artinian algebra. Then, by definition, $A$ is finite dimensional over $\bbC$.

\begin{defn}
	Let $A$ be a graded Artinian algebra.
	\begin{enumerate}	
		\item We say that $A$ has the maximal rank property (MRP) if for any $d$, the homomorphism induced by multiplication by $f$
		\[\times f\colon A_i\longrightarrow A_{i+d}\]
		has maximal rank for all $i$ (i.e., is injective or surjective), whenever $f$ is a general form of degree $d$.
		
		\item We say that $A$ has the strong Lefschetz property (SLP) if for any $d$, the homomorphism induced by multiplication by $\ell^d$
		\[\times\ell^d\colon A_i\longrightarrow A_{i+d}\]
		has maximal rank for all $i$ (i.e., is injective or surjective), whenever $\ell$ is a general linear form.
	\end{enumerate}
\end{defn}

\begin{remark}
The strong Lefschetz properties have been extensively investigated in the literature (see \cite{MiglioreNagel2013} and the references therein), while the maximal rank property has only been introduced in \cite{MiglioreMiro-Roig2003} by Migliore and Mir\'o-Roig. Moreover, it is easy to see that SLP implies MRP by semi-continuity.
\end{remark}

Both concepts are motivated by the following theorem which was proved by Stanley in \cite{Stanley1980} using algebraic topology, by Watanabe in \cite{Watanabe1987} using representation theory, by Reid, Roberts and Roitman in \cite{ReidRobertsRoitman1991} with algebraic methods.

\begin{thm}\label{SLP monomial complete intersection}\cite[Theorem 1.1]{MiglioreNagel2013}
	Let $R=\bbC[x_1,\dots,x_r]$, and let $I$ be the Artinian complete intersection $\langle x_1^{d_1},\dots,x_r^{d_r}\rangle$, where $d_i$'s are positive integers. Then $R/I$ has the SLP. 
\end{thm}

Let $\bbP^{n+1}$ be the $(n+1)$-dimensional complex projective space, and let $Y\subset\bbP^{n+1}$ be a hypersurface defined by a homogeneous polynomial $h$ of degree $d$. We denote by 
\[J(Y)=\langle \partial h/\partial x_0,\dots,\partial h/\partial x_{n+1}\rangle\]
the \emph{Jacobian ideal} of $Y$, where $[x_0\colon\dots\colon x_{n+1}]$ are the coordinates of $\bbP^{n+1}$. Then the \emph{Milnor algebra} of $Y$ is defined to be the graded $\bbC$-algebra
\[M(Y)\colon=\bbC[x_0,\dots,x_{n+1}]/J(Y).\]

\begin{remark}\cite[p.\,109]{Dimca1987}\label{HM series Milnor algebra}
	One observes that the Hilbert series of the Milnor algebra $M(Y)$ of a general degree $d$ hypersurface $Y$ in $\bbP^{n+1}$ is
	\[H(M(Y))(t)=(1+t+t^2+\dots+t^{d-2})^{n+2},\]
	where $\rho=(d-2)(n+2)$ is the top degree of $M(Y)$. The celebrated Macaulay's theorem (cf. \cite[Th\'eor\`eme 18.19]{Voisin2002}) says that the multiplication map
	\[\mu_{i,j} \colon\ M(Y)_i\times M(Y)_j\longrightarrow M(Y)_{i+j}\]
	is non-degenerated for $i+j\leq \rho$. Using the perfect pairing 
	\[M(Y)_i\times M(Y)_{\rho-i}\rightarrow M(Y)_{\rho}\cong \bbC,\]
	we see that the dimension of $M(Y)_i$ is symmetric. Recall that an element $f\in M(Y)$ of degree $j$ is called faithful if the multiplication $\times f\colon M(Y)_i\rightarrow M(Y)_{i+j}$ has maximal rank for all $i$. Since the dimension of $M(Y)_i$ is strictly increasing over the interval $[0,\rho/2]$,	an element $f$ of degree $j$ is faithful if and only if it induces injections $M(Y)_i\rightarrow M(Y)_{i+j}$ for $i\leq (\rho-j)/2$, equivalently it induces surjections $M(Y)_i\rightarrow M(Y)_{i+j}$ for $i\geq (\rho-j)/2$.
\end{remark}

The proof of Theorem \ref{Restrition Invariant Picard Group} relies on the nonexistence of certain twisted vector fields over $X$. To prove this, we reduce the problem to the nonexistence of certain twisted vector fields over $Y$ by proving an extension result (cf. Theorem \ref{Globalsections}). The main ingredient of the proof of Theorem \ref{Globalsections} is the SLP of the Milnor algebra $M(Y)$ which is well-known to experts. Recall that the Fermat hypersurface of degree $d$ in $\bbP^{n+1}$ is defined by the equation $x_0^d+\dots+x_{n+1}^d=0$.

\begin{prop}\label{MRP}
	Let $Y\subset \bbP^{n+1}$ be a general hypersurface of degree $d$. Then the Milnor algebra $M(Y)$ of $Y$ has the SLP. In particular, $M(Y)$ has the MRP.
\end{prop}

\begin{proof}
	Thanks to Theorem \ref{SLP monomial complete intersection}, the Milnor algebra of the Fermat hypersurface of degree $d$ in $\bbP^{n+1}$ has the SLP. Then we conclude by semi-continuity.
\end{proof}

\section{Twisted $(n-1)$-forms and special cohomologies}\label{Section Vanishing theorems}

In this section, we collect some vanishing results about the cohomologies of twisted $(n-1)$-forms of $n$-dimensional irreducible Hermitian symmetric spaces of compact type. Moreover, we introduce the notion of special cohomologies and we prove that all irreducible Hermitian symmetric spaces of compact type have special cohomologies. This notion is very useful in studying the twisted vector fields over complete intersections in Hermitian symmetric spaces in the next section. We start with a result due to Snow.

\begin{defn}\cite[Definition 2.4]{BiswasChaputMourougane2015}
	Let $\ell, n\in\bbN$ be two fixed positive integers. An $n$-tuple of integers $\textbf{a}_n=(a_i)_{1\leq i\leq n}$ is called an $\ell$-admissible $C_n$-sequence if $\vert a_i\vert=i$ and $a_i+a_j\not=2\ell$ for all $i\leq j$. Its weight is defined to be $p(\textbf{a}_n)=\sum_{a_i>0} a_i$ and its $\ell$-cohomological degree is defined to be
	\[q(\textbf{a}_n)=\#\{\, (i,j)\, \mid\, i\leq j\ \textrm{and}\ a_i+a_j>2\ell\, \}.\]
\end{defn}

\begin{prop}\cite[\S 2.1]{Snow1988}\label{Lagrange admissible sequence}
	Let $M=Sp(2n)/U(n)$ be a type $C_n$ irreducible Hermitian symmetric space of compact type. Then $H^q(M,\Omega_M^p(\ell))\not=0$ implies that there exists an $\ell$-admissible $C_n$-sequence such that its weight is $p$ and its $\ell$-cohomological degree is $q$. 
\end{prop}

\begin{example}\label{Example C4}
	Denote by $M$ the Lagrangian Grassmannian $Sp(8)/U(4)$. Then $M$ is a $10$-dimensional Fano manifold with index $5$. Moreover, if $\ell$ is an integer such that $1\leq \ell\leq 4$, then we have $H^q(M,\Omega_M^{9}(\ell))=0$ for any $q\geq 0$. In fact, if $H^q(M,\Omega_M^9(\ell))\not=0$, by Proposition \ref{Lagrange admissible sequence}, there exists an $\ell$-admissible $C_4$-sequence $\textbf{a}$ with $\ell$-cohomological degree $q$ and weight $9$. This implies
	\[\textbf{a}=(-1,2,3,4).\]
	As $1\leq \ell\leq 4$, one can easily see that $\textbf{a}$ cannot be $\ell$-admissible.
\end{example}

Before giving the statement in the general case, we recall the cohomologies of the twisted holomorphic $p$-forms of projective spaces and smooth quadric hypersurfaces.

\begin{thm}\cite{Bott1957}\label{Bott Formula}
	Let $n$, $p$, $q$ and $\ell$ be integers, with $n$ positive and $p$ and $q$ nonnegative. Then
	\[h^q(\bbP^n,\Omega_{\bbP^n}^p(\ell))=\begin{dcases}
	\binom{n+\ell-p}{\ell}\binom{\ell-1}{p}, & \textrm{if}\ q=0, 0\leq p\leq n, \ell>p;\\
	1, &\textrm{if}\ \ell=0, p=q;\\
	\binom{p-\ell}{-\ell}\binom{-\ell-1}{n-p}, &\textrm{if}\ q=n, 0\leq p\leq n, \ell<p-n;\\
	0, &\textrm{otherwise}.
	\end{dcases}\]
\end{thm}

As a consequence, $H^q(\bbP^n,\Omega_{\bbP^n}^{n-1}(\ell))\not=0$ for some $\ell\in\bbZ$ if and only if $q=0$ and $\ell\geq n$, or $q=n-1$ and $\ell=0$, or $q=n$ and $\ell\leq -2$.

\begin{thm}\cite[Theorem 4.1]{Snow1986}\label{Quadric Cohomologies}
	Let $X$ be an $n$-dimensional smooth quadric hypersurface.
	\begin{enumerate}
		\item If $-n+p\leq\ell\leq p$ and $\ell\not=0$, $-n+2p$, then $H^q(X,\Omega_X^{p}(\ell))=0$ for all $q$.
		
		\item $H^q(X,\Omega_X^p)\not=0$ if and only if $q=p$.
		
		\item $H^q(X,\Omega_X^p(-n+2p))\not=0$ if and only if $p+q=n$.
		
		\item If $\ell>p$, then $H^q(X,\Omega^p_X(\ell))\not=0$ if and only if $q=0$.
		
		\item If $\ell<-n+p$, then $H^q(X,\Omega_X^p(\ell))\not=0$ if and only if $q=n$.
	\end{enumerate}
\end{thm}

In particular, if $X$ is a smooth quadric hypersurface of dimension $n$, then $H^q(X,\Omega_X^{n-1}(\ell))\not=0$ for some $\ell\in\bbZ$ if and only if $q=0$ and $\ell\geq n$, or $q=1$ and $\ell=n-2$, or $q=n-1$ and $\ell=0$, or $q=n$ and $\ell\leq -2$. The following general result is essentially proved in \cite{Snow1986} and \cite{Snow1988}.

\begin{prop}\label{Vanishing p=n-1}
	Let $M$ be an $n$-dimensional irreducible  Hermitian symmetric space of compact type. Then $H^q(M,\Omega_M^{n-1}(\ell))\not=0$ if and only if one of the following conditions is satisfied.
	\begin{enumerate}
		\item $q=0$ and $\ell\geq \min\{n,r_M\}$.
		
		\item $q=n-1$ and $\ell=0$.
		
		\item $q=n$ and $\ell\leq -2$.
		
		\item $M\cong Q^n$, $q=1$ and $\ell=n-2$.
	\end{enumerate}
\end{prop}

\begin{proof}
	If $n\geq r_M$ or $n\leq 3$, then $X$ is isomorphic to $\bbP^n$ or $Q^n$ and we can conclude by Theorem \ref{Bott Formula} and Theorem \ref{Quadric Cohomologies}. On the other hand, if $\ell\geq r_M$, by \cite[Proposition 1.1]{Snow1988}, the cohomological degree of $\Omega_M^{n-1}(\ell)$ is $0$. As a consequence, $H^q(M,\Omega_M^{n-1}(\ell))\not=0$ if and only if $q=0$. Moreover, it is well-known that $H^q(M,\Omega_M^p)\not=0$ if and only if $q=p$. So we shall assume that $n-1\geq r_M\geq \ell+1$, $n\geq 4$ and $\ell\not=0$. In particular, $M$ is not of type $B_n$.
	
	If $\ell\leq -2$, by Serre duality, we obtain that $H^q(M,\Omega_M^{n-1}(\ell))\not=0$ if and only if $H^{n-q}(M,\Omega_M^1(-\ell))\not=0$. Recall that the cohomological degree of the sheaf $\Omega_M^1(-\ell)$ is $0$ if $-\ell\geq 2$ by \cite[Proposition 1.1]{Snow1988}. So $H^q(M,\Omega_M^{n-1}(\ell))\not=0$ if and only if $q=n$ if $\ell\leq -2$.
	
	If $\ell=-1$, by Serre duality again, we see that  $H^q(M,\Omega_M^{n-1}(-1))\not=0$ if and only if $H^{n-q}(M,\Omega_M^1(1))\not=0$. Thanks to \cite[Theorem 2.3]{Snow1986}, we have $H^{n-q}(M,\Omega_M^1(1))=0$ for all $q\geq 0$ if $M$ is not of type $C_n$. On the other hand, if $M$ is of type $C_n$, the vanishing of $H^q(M,\Omega_M^1(1))$ follows from \cite[Theorem 2.3]{Snow1988}.
	
	If $1\leq \ell\leq r_M-1$, we can prove the result case by case. If $M$ is of type $E_6$ or $E_7$, from \cite[Table 4.4 and Table 4.5]{Snow1988}, we have $H^q(M,\Omega_M^{n-1}(\ell))=0$ for any $q\geq 0$. If $M$ is of type $A_n$, as $M$ is not isomorphic to $\bbP^n$ or $Q^n$, we get $H^q(M,\Omega_M^{n-1}(\ell))=0$ for all $q\geq 0$ by \cite[Theorem 3.4 (3)]{Snow1986}. Here we remark that $Gr(2,4)$ is isomorphic to $Q^4$. If $M$ is of type $C_n$ and $n\not=4$, we have $H^q(M,\Omega_M^{n-1}(\ell))=0$ for all $q\geq 0$ by \cite[Theorem 2.4 (3)]{Snow1988}. If $M$ is of type $C_4$, then $M$ is isomorphic to the $10$-dimensional homogeneous space $Sp(8)/U(4)$, and we get $H^q(M,\Omega_M^{9}(\ell))=0$ for all $q\geq 0$ according to Example \ref{Example C4}. If $M$ is of type $D_n$, it follows from \cite[Theorem 3.4 (3)]{Snow1988} that $H^q(M,\Omega_M^{n-1}(\ell))=0$ for all $q\geq 0$ if $n\geq 5$. If $M$ is of type $D_n$ and $n\leq 4$, then $M$ is isomorphic to either $\bbP^n$ or $Q^n$, which is impossible by our assumption.
\end{proof}

As a direct application, we get the following result which is useful to describe the twisted vector fields over complete intersections.

\begin{cor}\label{p+q=n}
	For $n\geq 3$, let $M$ be an $n$-dimensional irreducible Hermitian symmetric space of compact type. Then $H^{n-1}(M,\Omega_M^1(\ell))\not=0$ if and only if $\ell=-n+2$ and $M$ is isomorphic to a smooth quadric hypersurface $Q^n$.
\end{cor}

\begin{proof}
	As $H^{n-1}(M,\Omega_M^1(\ell))\not=0$ if and only if $H^1(M,\Omega_M^{n-1}(-\ell))\not=0$ by Serre duality, thus the result follows from Proposition \ref{Vanishing p=n-1} directly.
\end{proof}

Moreover, one can easily derive the following result for smooth hypersurfaces in projective spaces by Bott's formula.

\begin{lemma}\label{p+q=n Hypersurfaces}
	For $n\geq 3$, let $Y\subset\bbP^{n+1}$ be a smooth hypersurface of degree $d$. Then we have $H^{n-1}(Y,\Omega_Y^1(-r_Y+t))=0$ for $t>d$, where $r_Y=n+2-d$.
\end{lemma}

\begin{proof}
	By Bott's formula (cf. Theorem \ref{Bott Formula}) and the following exact sequence of sheaves
	\[0\rightarrow \Omega^1_{\bbP^{n+1}}(-r_Y+t-d)\rightarrow\Omega_{\bbP^{n+1}}^1(-r_Y+t)\rightarrow\Omega_{\bbP^{n+1}}^1(-r_Y+t)\vert_Y\rightarrow 0,\]
	we see that $H^{n-1}(Y,\Omega_{\bbP^{n+1}}^1(-r_Y+t)\vert_Y)=0$ for any $t\in\bbZ$.	Thus, the following exact sequence of $\cO_Y$-sheaves
	\[0\rightarrow \cO_Y(-r_Y+t-d)\rightarrow \Omega_{\bbP^{n+1}}^1(-r_Y+t)\vert_Y\rightarrow \Omega_Y^1(-r_Y+t)\rightarrow 0\]
	induces an injective map of groups
	\[H^{n-1}(Y,\Omega^1_Y(-r_Y+t))\rightarrow H^n(Y,\cO_Y(-r_Y+t-d)).\]
	Then we can conclude by Kodaira's vanishing theorem.
\end{proof}

\begin{defn}
	Let $(Z,\cO_Z(1))$ be a polarized projective manifold of dimension $n\geq 4$. We say that $(Z,\cO_Z(1))$ has special cohomologies if $H^q(Z,\Omega_Z^1(\ell))=0$ for all $\ell\in\bbZ$ and all $2\leq q\leq n-2$.
\end{defn}

We remark that our definition of special cohomologies is much weaker than that given in \cite{PeternellWisniewski1995}.

\begin{example}\label{Complete intersection projective spaces Special Cohomologies}
	By \cite[Corollary 2.3.1]{Naruki1977/78}, an $n$-dimensional smooth complete intersection $Y$ in a projective space has special cohomologies if $n\geq 4$. Moreover, if $\widetilde{Y}$ is a cyclic covering of $Y$, then $Y$ has special cohomologies (see \cite[Theorem 1.6]{PeternellWisniewski1995}). 
\end{example}

\begin{example}
	Let $Y$ be a smooth weighted complete intersection of dimension $n$ in a weighted projective space, and let $\cO_Y(1)$ be the restriction to $Y$ of the universal $\cO(1)$-sheaf from the weighted projective space. Then $(Y,\cO_Y(1))$ has special cohomologies by \cite[Satz 8.11]{Flenner1981} .
\end{example}

\begin{prop}\label{symmetric space has special cohomology}
	Let $(M,\cO_M(1))$ be an $n$-dimensional irreducible Hermitian symmetric space of compact type. If $n\geq 4$, then $(M,\cO_M(1))$ has special cohomologies.
\end{prop}

\begin{proof}
	By Serre duality, it suffices to consider the group $H^{n-q}(M,\Omega_M^{n-1}(-\ell))$. As $2\leq q\leq n-2$, we get $2\leq n-q\leq n-2$. Now it follows from Proposition \ref{Vanishing p=n-1}.
\end{proof}

\section{Extension of twisted vector fields}\label{Section Extension of vector fields}

This section is devoted to study global twisted vector fields over smooth complete intersections in an irreducible Hermitian symmetric space of compact type. The main aim is to show that the global twisted vector fields over complete intersections can be extended to be global twisted vector fields over the ambient space (cf. Theorem \ref{surjectivity twisted global vector fields}).

\subsection{Twisted vector fields over complete intersections} 

Let $(Z,\cO_Z(1))$ be a polarized manifold, and let $Y\subset Z$ be a subvariety. Then we have a natural restriction map
\[\rho_t\colon H^0(Z,T_Z(t))\longrightarrow H^0(Y,T_Z(t)\vert_Y)\]
for any $t\in\bbZ$. The following proposition concerns on the surjectivity of $\rho_t$ and its proof was communicated to me by the anonymous referee.

\begin{prop}\label{Extension of vector field}
	Let $(Z,\cO_Z(1))$ be a polarized projective manifold. Let $Y$ be the complete intersection $H_1\cap\dots\cap H_r$ where $H_i\in\vert \cO_Z(d_i)\vert$. Assume moreover that there exists an integer $r_Z\in\bbZ$ such that $\cO_Z(-K_Z)\cong\cO_Z(r_Z)$. Given $t\in\bbZ$, if $(Z,\cO_Z(1))$ has special cohomologies, $\dim(Y)\geq 2$ and $H^{\dim(Z)-1}(Z,\Omega^1_Z(-r_Z+d_i-t))=0$ for all $1\leq i\leq r$, then the natural restriction 
	\[\rho_t\colon H^0(Z,T_Z(t))\longrightarrow H^0(Y,T_Z(t)\vert_{Y})\]
	is surjective.
\end{prop}

\begin{proof}
	Denote by $E$ the vector bundle $\cO_Z(-d_1)\oplus\dots\oplus\cO_Z(-d_r)$ over $Z$. Consider the following Koszul resolution of $\cO_Y$
	\[0\rightarrow \wedge^r E\rightarrow \wedge^{r-1}E\rightarrow\dots\rightarrow E\rightarrow\cO_Z\rightarrow\cO_Y\rightarrow 0.\]
	Tensoring it with $T_Z(t)$, we obtain the following exact sequence
	\begin{multline}\label{Koszul complexe vector field}
		0\rightarrow T_Z(t)\otimes(\wedge^r E)\rightarrow T_Z(t)\otimes(\wedge^{r-1}E)\rightarrow\\
		\dots\rightarrow T_Z(t)\otimes E\rightarrow T_Z(t)\rightarrow T_Z(t)\vert_Y\rightarrow 0.
	\end{multline}
	By definition, for any $1\leq j\leq r$, the vector bundle $\wedge^j E$ splits into a direct sum of line bundles as $\oplus_i\cO_Z(-d_{ij})$ with $d_{ij}\geq j$.
	Since $(Z,\cO_Z(1))$ has special cohomologies and $r=\codim(Y)\leq n-2$, by Serre duality, we have 
	\[H^j(Z,T_Z(t)\otimes(\wedge^j E))\cong\bigoplus_{i} H^q(Z,T_Z(t-d_{ij}))=0\]
	for any $j\geq 2$. On the other hand, as $H^{\dim(Z)-1}(Z,\Omega^1_Z(-r_Z+d_i-t))=0$ by our assumption, by Serre duality again, 
	we obtain 
	\[H^1(Z,T_Z(t)\otimes E)\cong\bigoplus_{i=1}^r H^{\dim(Z)-1}(Z,\Omega^1_Z(-r_Z+d_i-t)=0.\]
	Then the second quadrant spectral sequence associated to the complex \eqref{Koszul complexe vector field} (see \cite[Appendix B]{Lazarsfeld2004}) implies that the natural restriction 
	\[H^0(Z,T_Z(t))\rightarrow H^0(Y,T_Z(t)\vert_Y)\]
	is surjective.
\end{proof}

As an immediate application, we derive the following theorem which will play a key role in the proof of Theorem \ref{Effective restriction HSS}.

\begin{thm}\label{surjectivity twisted global vector fields}
	Let $M$ be an $(n+r)$-dimensional irreducible Hermitian symmetric space of compact type which is not isomorphic to any smooth quadric hypersurface $Q^{n+r}$. Let $Y=H_1\cap\dots\cap H_r$ be a complete intersection of dimension $n$ where $H_i\in\vert \cO_M(d_i)\vert$. Assume moreover that $\dim(Y)=n\geq 2$. Then the natural restriction
	\[\rho_t\colon H^0(M,T_M(t))\longrightarrow H^0(Y,T_M(t)\vert_{Y}).\]
	is surjective for any $t\in\bbZ$.
\end{thm}		

\begin{proof}
	If $n+r\geq 4$, this follows from Corollary \ref{p+q=n}, Proposition \ref{symmetric space has special cohomology} and Proposition \ref{Extension of vector field}. If $n+r=3$, then $Y$ is a hypersurface of degree $d$ in $M$. Thanks to the following exact sequence
	\[0\rightarrow T_M(t-d)\rightarrow T_M(t)\rightarrow T_M(t)\vert_Y\rightarrow 0,\]
	to prove the surjectivity of $\rho_t$, it is enough to show that $H^1(M,T_M(t-d))=0$. By Serre duality, it is equivalent to showing that $H^2(M,\Omega_M^1(-r_M+d-t))=0$. Since $M$ is not isomorphic to any smooth quadric hypersurface, we conclude by Corollary \ref{p+q=n}.
\end{proof}

If $M$ is a smooth quadric hypersurface, then we can also regard $Y$ as a complete intersection of degree $(2,d_1,\dots,d_r)$ in the projective space $\bbP^{n+r+1}$. In general, we have the following result.

\begin{thm}\label{Surjective complete intersection Projective spaces}
	Let $Y=H_1\cap \dots \cap H_r$ be a complete intersection where $H_i\in\vert \cO_{\bbP^{n+r}}(d_i)\vert$. Assume moreover that $H_1$ is smooth, $\dim(Y)=n\geq 2$ and $n+r\geq 4$. If $d_i\geq 2$ for all $1\leq i\leq r$ and $t$ is an integer such that $d_i-t>d_1$ for any $2\leq i\leq r$, then the natural restriction
	\[\rho_t\colon H^0(H_1,T_{H_1}(t))\rightarrow H^0(Y,T_{H_1}(t)\vert_{Y})\]
	is surjective.
\end{thm}

\begin{proof}
	If $r=1$, the statement is trivial. So we may assume that $r\geq 2$. If $n+r\geq 5$, by definition and \cite[Corollary 2.3.1]{Naruki1977/78}, the hypersurface $H_1$ has special cohomologies (cf. Example \ref{Complete intersection projective spaces Special Cohomologies}). Thanks to Proposition \ref{Extension of vector field}, it suffices to verify that we have
	\[H^{\dim(Y_1)-1}(Y_1,\Omega_{Y_1}^1(-r_{Y_1}+d_i-t))=0\]
	for all $2\leq i\leq r$. Since $t\in\bbZ$ is an integer such that $d_i-t>d_1$ for any $2\leq i\leq r$, it follows from Lemma \ref{p+q=n Hypersurfaces} immediately. 
	
	If $n+r=4$, then $Y=H_1\cap H_2$. As $d_2-t>d_1$ by assumption, thanks to Lemma \ref{p+q=n Hypersurfaces}, we have $H^2(H_1,\Omega_{H_1}^1(-r_{H_1}+d_2-t))=0$. By Serre duality, we obtain $H^1(H_1,T_{H_1}(t-d_2))=0$. Then, it follows from the following exact sequence
	\[0\rightarrow T_{H_1}(t-d_2)\rightarrow T_{H_1}(t)\rightarrow T_{H_1}(t)\vert_Y\rightarrow 0\]
	that the map $H^0(H_1,T_{H_1}(t))\rightarrow H^0(Y,T_{H_1}(t)\vert_Y)$ is surjective. 
\end{proof}

\subsection{Twisted vector fields over hypersurfaces in projective spaces}
The global sections of $T_{\bbP^{n}}(t)$ can be expressed explicitly by homogeneous polynomials of degree $t+1$. To see this, we consider the twisted Euler sequence
\[0\rightarrow \cO_{\bbP^{n}}(t) \rightarrow \cO_{\bbP^{n}}(t+1)^{\oplus (n+1)}\rightarrow T_{\bbP^{n}}(t)\rightarrow 0.\]
Using the fact $H^1(\bbP^{n},\cO_{\bbP^{n}}(t))=0$, we see that the restriction map 
\[H^0(\bbP^{n},\cO_{\bbP^{n}}(t+1)^{\oplus(n+1)})\longrightarrow H^0(\bbP^{n},T_{\bbP^{n}}(t))\]
is surjective, so a global section $\sigma$ of $T_{\bbP^{n}}(t)$ is given by a vector field on the affine complex vector space $\bbC^{n+1}$
\[\sigma=f_0\frac{\partial}{\partial x_0}+\dots+f_{n}\frac{\partial}{\partial x_{n}},\]
where $f_i$'s are homogeneous polynomials of degree $t+1$. Let $Y\subset\bbP^{n}$ be a smooth hypersurface defined by a homogeneous polynomial $h$ and let $X$ be a submanifold of $Y$. Then the restriction $\sigma\vert_X$ is a global section of $T_Y(t)\vert_X$ if and only if we have
\[\left(\left.f_0\frac{\partial h}{\partial x_0}+\dots+f_{n}\frac{\partial h}{\partial x_{n}}\right)\right\vert_X\equiv 0.\]
Furthermore, we have $\sigma\vert_X\equiv 0$ if and only if
\[\left(x_if_j-x_jf_i\right)\vert_X\equiv 0,\ \ 0\leq i\leq j\leq n.\]

Let $Y$ be a general complete intersection in an $N$-dimensional irreducible Hermitian symmetric space $M$ of compact type such that $N\geq 3$, and let $X\in\vert\cO_Y(d)\vert$ be a general hypersurface of $Y$. By \cite{Wahl1983}, $H^0(M,T_M(t))\not=0$ for some $t<0$ if and only if $M\cong\bbP^{N}$ and $t=-1$. According to Theorem \ref{surjectivity twisted global vector fields} and Theorem \ref{Surjective complete intersection Projective spaces}, we get
\[H^0(Y,T_Y(t))=H^0(X,T_Y(t)\vert_X)=H^0(M,T_M(t))=0\]
for any $t\leq -2$. In the following theorem, we generalize this result to show that if $Y$ is a general hypersurface of $\bbP^{n+1}$, then the natural restriction
\[\alpha_t\colon H^0(Y,T_Y(t))\longrightarrow H^0(X,T_Y(t)\vert_X)\]
is surjective for $t\leq t_0$ large enough depending only on $n$ and the degrees of $X$ and $Y$. This theorem is a key ingredient of the proof of Theorem \ref{Restrition Invariant Picard Group}.

\begin{thm}\label{Globalsections}
	Let $Y\subset \bbP^{n+1}$ be a general smooth hypersurface defined by the homogeneous polynomial $h$ of degree $d_h\geq 2$ and let $X\in\vert\cO_Y(d)\vert$ be a general smooth divisor. If $n\geq 3$, then the restriction map 
	\[H^0(Y,T_Y(t))\longrightarrow H^0(X,T_Y(t)\vert_X)\]
	is surjective for any $t\leq (\rho+d)/2-d_h$, where $\rho=(d_h-2)(n+2)$ is the top degree of the Milnor algebra of $Y$.
\end{thm}

\begin{proof}
	Since the natural restriction $H^0(\bbP^{n+1},\cO_{\bbP^{n+1}}(d))\rightarrow H^0(Y,\cO_Y(d))$ is surjective, there exists a general homogeneous polynomial $f$ of degree $d$ such that $X=\{f=h=0\}$. We denote by $M(Y)$ and $J(Y)$ the Milnor algebra and Jacobian ideal of $Y$, respectively. Since $H^0(X,T_Y(t)\vert_X)$ is a subset of $H^0(X,T_{\bbP^{n+1}}(t)\vert_X)$ and $H^0(\bbP^{n+1},T_{\bbP^{n+1}}(t))=0$ for $t\leq -2$, by Theorem \ref{surjectivity twisted global vector fields}, we may assume $t\geq -1$. 
	
	Let $s\in H^0(X,T_Y(t)\vert_X)$ be a global section. By Theorem \ref{surjectivity twisted global vector fields}, the section $s$ is the restriction of some global section $\sigma\in H^0(\bbP^{n+1},T_{\bbP^{n+1}}(t))$. Therefore there exist some polynomials $f_i$ of degree $t+1$ such that
	\[s=\sigma\vert_X=\left.\left(f_0\frac{\partial}{\partial x_0}+\dots+f_{n+1}\frac{\partial}{\partial x_{n+1}}\right)\right\vert_X\]
	and
	\[\left(\left. f_0\frac{\partial h}{\partial x_0}+\dots+f_{n+1}\frac{\partial h}{\partial x_{n+1}}\right)\right\vert_X=0.\]
	As a consequence, there exist two homogeneous polynomials $g$ and $p$ (maybe zero) such that
	\[f_0\frac{\partial h}{\partial x_0}+\dots+f_{n+1}\frac{\partial h}{\partial x_{n+1}}=gf+ph.\]
	We claim that $g$ is contained in the Jacobian ideal $J(Y)$ of $Y$. In fact, by Euler's homogeneous function theorem, it follows
	\[\left(f_0-\frac{1}{d_h}px_0\right)\frac{\partial h}{\partial x_0}+\dots+\left(f_{n+1}-\frac{1}{d_h}px_{n+1}\right)\frac{\partial h}{\partial x_{n+1}}=gf.\]
	Thanks to Proposition \ref{MRP}, the Milnor algebra $M(Y)$ has maximal rank property, hence, by the genericity assumption of $f$, the multiplication map 
	\[(\times f)\colon M(Y)_{t+d_h-d}\longrightarrow M(Y)_{t+d_h}\]
	has maximal rank. Moreover, by the assumption, we have
	\[t+d_h-d\leq \frac{\rho-d}{2},\]
	so the multiplication map $(\times f)$ is injective (cf. Remark \ref{HM series Milnor algebra}). It follows that $g=0$ in $M(Y)$, or equivalently, the polynomial $g$ is contained in the Jacobian ideal of $Y$. Therefore there exist some homogeneous polynomials $g_i$'s of degree $t-d+1$ such that 
	\[g=g_0\frac{\partial h}{\partial x_0}+\dots+g_{n+1}\frac{\partial h}{\partial x_{n+1}}.\]
	This yields
	\[\left(f_0\frac{\partial h}{\partial x_0}+\dots+f_{n+1}\frac{\partial h}{\partial x_{n+1}}\right)-\left(g_0f\frac{\partial h}{\partial x_0}+\dots+g_{n+1}f\frac{\partial h}{\partial x_{n+1}}\right)=ph.\]
	We denote by $\sigma'\in H^0(\bbP^{n+1},T_{\bbP^{n+1}}(t))$ the global section defined by
	\[g_0f\frac{\partial}{\partial x_0}+\dots+g_{n+1}f\frac{\partial}{\partial x_{n+1}}.\]
	Then $(\sigma-\sigma')\vert_Y\in H^0(Y,T_Y(t))$. Moreover, as $\sigma'\vert_X\equiv 0$, we have 
	\[(\sigma-\sigma')\vert_X=\sigma\vert_X=s.\]
	Hence the restriction map
	\[H^0(Y,T_Y(t))\rightarrow H^0(X,T_Y(t)\vert_X)\]
	is surjective.
\end{proof}

\section{Stability and effective restriction with invariant Picard group}\label{Section Stability Invariant Picard group}

This section is devoted to study the stability of tangent bundles of smooth complete intersections in Hermitian symmetric spaces. As mentioned in the introduction, this problem was studied by Peternell and Wi{\'s}niewski in \cite{PeternellWisniewski1995} in the projective spaces case. Moreover, we will also consider the effective restriction problem for tangent bundles.

\subsection{Vanishing theorem and stability of tangent bundles}

We start with an observation whose statement was suggested by the referee. It is very useful when we consider the cohomologies of smooth complete intersections in some projective manifolds with many cohomological vanishings.  

\begin{lemma}\label{cohomology hypersurface to ambiant space}
	For $n\geq 2$, let $(Z,\cO_Z(1))$ be a polarized projective manifold of dimension $n+r$. Let $Y=H_1\cap\dots\cap H_r$ be a smooth complete intersection of dimension $n$ where $H_i\in\vert\cO_Z(d_i)\vert$. If there exist integers $p$, $q$ and $\ell$ such that $p+q\leq n$ and $H^q(Y,\Omega_Y^p(\ell))\not=0$, then there exist integers $j_1,\dots,j_r\in\bbN$ such that $0\leq j_1+\dots+j_r\leq p-1$ and such that one of the following statements holds.
	\begin{enumerate}
		\item There exists an integer $2\leq s\leq r+1$ such that \[H^{q+p}(Y,\cO_Y(\ell-pd_{s-1}-(d_s-d_{s-1})j_s-\dots-(d_r-d_{s-1})j_r)\not=0.\]
		
		\item $H^{q+j_1+\dots+j_r}(Y,\Omega_Z^{p-j_1-\dots-j_r}(\ell-j_1 d_1-\dots-j_r d_r)\vert_Y)\not=0$.
	\end{enumerate}
\end{lemma}

\begin{proof}
	 Let us denote by $Y_i$ $(1\leq i\leq r)$ the scheme-theoretic complete intersection $H_1\cap\dots \cap H_i$. Then we have $Y=Y_r$. Moreover, for convenience, we will denote $Z$ by $Y_0$. Since $Y$ is smooth, each $Y_i$ is smooth in a neighborhood of $Y$. In particular, the cotangent sheaf $\Omega_{Y_i}^1$ of $Y_i$ is locally free in a neighborhood of $Y$. Let $s$ be the minimal positive integer such that there exist integers $j_s,\dots,j_r\in\bbN$ such that $0\leq k\leq p-1$ where $k=j_s+\dots+j_r$ and such that
	 \[H^{q+k}(Y,\Omega_{Y_{s-1}}^{p-k}(\ell-j_s d_s-\dots-j_r d_r)\vert_Y)\not=0.\]
	 Clearly we have $1\leq s\leq r+1$ and we are in Case (2) if $s=1$. Now we assume that $s\geq 2$. Then, for any $j\in\bbN$ such that $0\leq j+k\leq p-1$, we have
	 \begin{equation}\label{Inductive vanishing}
	 	H^{q+j+k}(Y,\Omega_{Y_{s-2}}^{p-j-k}(\ell-jd_{s-1}-j_s d_s-\dots-j_r d_r)\vert_Y)=0.
	 \end{equation}
	 The restriction over $Y$ of the conormal sequence of $Y_{s-1}$ in $Y_{s-2}$
	 \[0\rightarrow \cO_Y(-d_{s-1})\rightarrow \Omega^1_{Y_{s-2}}\vert_Y\rightarrow\Omega^1_{Y_{s-1}}\vert_Y\rightarrow 0\]
	 induces an exact sequence of vector bundles
	 \[0\rightarrow \Omega_{Y_{s-1}}^{p-j-k-1}(-d_{s-1})\vert_Y\rightarrow\Omega_{Y_{s-2}}^{p-j-k}\vert_Y\rightarrow\Omega_{Y_{s-1}}^{p-j-k}\vert_Y\rightarrow 0.\]
    Tensoring it with $\cO_Y(\ell-jd_{s-1}-j_{s} d_{s}-\dots-j_r d_r)$, we obtain
    	\begin{multline*}
    		0\rightarrow \Omega_{Y_{s-1}}^{p-j-k-1}(\ell-(j+1)d_{s-1}-j_{s} d_{s}-\dots-j_r d_r)\vert_Y\\
    		\rightarrow\Omega_{Y_{s-2}}^{p-j-k}(\ell-jd_{s-1}-j_{s} d_{s}-\dots-j_r d_r)\vert_Y\\
    		\rightarrow\Omega_{Y_{s-1}}^{p-j-k}(\ell-jd_{s-1}-j_{s} d_{s}-\dots-j_r d_r)\vert_Y\rightarrow 0.
    	\end{multline*}
	 Then \eqref{Inductive vanishing} shows that for any $j\in\bbN$ such that $0\leq j+k\leq p-1$, the induced map
	 \begin{multline*}
	 	H^{q+j+k}(Y,\Omega_{Y_{s-1}}^{p-j-k}(\ell-jd_{s-1}-j_{s} d_{s}-\dots-j_r d_r)\vert_Y)
	 	\rightarrow\\ H^{q+j+k+1}(Y,\Omega_{Y_{s-1}}^{p-j-k-1}(\ell-(j+1)d_{s-1}-j_{s} d_{s}-\dots-j_r d_r)\vert_Y)
	 \end{multline*}
	 is injective. Then the assumption 
	 \[H^{q+k}(Y,\Omega_{Y_{s-1}}^{p-k}(\ell-j_s d_s-\dots-j_r d_r)\vert_Y)\not=0.\]
	 implies
	 \[H^{q+p}(Y,\cO_Y(\ell-(p-k)d_{s-1}-j_s d_s-\dots-j_rd_r)\vert_Y)\not=0.\]
	 This completes the proof.
\end{proof}

Now we are in the position to prove Theorem \ref{Vanishing complete intersection}. The idea is to relate the cohomologies of $Y$ to the cohomologies of the ambient space $M$ using Lemma \ref{cohomology hypersurface to ambiant space} and the Koszul resolution.

\begin{proof}[Proof of Theorem \ref{Vanishing complete intersection}]
	As $M$ is a Fano manifold, by Kobayshi-Ochiai's theorem (see \cite{KobayashiOchiai1973}), we have $r_M\leq n+r+1$ with equality if and only if $M\cong\bbP^{n+r}$. On the other hand, if $M\cong\bbP^{n+r}$, thanks to \cite[Corollary 2.3.1]{Naruki1977/78}, under our assumption, we have $H^q(Y,\Omega_Y^p(\ell))\not=0$ if and only if $q=p$ and $\ell=0$. Hence the result holds if $r_M=n+r+1$. From now on, we shall assume that $r_M\leq n+r$. As a consequence, we have $r_Y\leq n-1$ since $d_i\geq 2$. Let us denote by $E$ the vector bundle $\cO_{M}(-d_1)\oplus\dots\oplus\cO_M(-d_r)$. Then we have the Koszul resolution of $\cO_Y$
	\begin{equation}\label{Koszul resolution}
		0\rightarrow \wedge^r E\rightarrow\wedge^{r-1} E\rightarrow\dots\rightarrow E\rightarrow\cO_M\rightarrow\cO_Y\rightarrow 0.
	\end{equation}
	Moreover, for any $1\leq j\leq r$, the vector bundle $\wedge^j E$ splits into a direct sum of line bundles as $\oplus_i\cO_M(-d_{ij})$ with $d_{ij}\geq 2j$.
	
	\textit{Proof of (1).} As $H^q(Y,\Omega_Y^p)\not=0$, thanks to Lemma \ref{cohomology hypersurface to ambiant space}, there exist integers $j_1,\dots,j_r\in\bbN$ such that $0\leq k\leq p-1$ where $k=j_1+\dots+j_r$ and such that we have either
	\begin{equation}\label{Cohom.Diff.l=0}
		H^{q+k}(Y,\Omega_M^{p-k}(-j_1 d_1-\dots-j_r d_r)\vert_Y)\not=0
	\end{equation}
	or
	\begin{equation}\label{Cohom.irreg.l=0}
		H^{q+p}(Y,\cO_Y(-pd_{s-1}-(d_r-d_{s-1})j_r-\dots-(d_s-d_{s-1})j_s)\not=0
	\end{equation}
	for some $2\leq s\leq r+1$. Note that since we have
	\begin{multline*}
		-pd_{s-1}-(d_r-d_{s-1})j_r-\dots-(d_s-d_{s-1})j_s\\
		=-j_s d_s-\dots-j_r d_r-(p-k)d_{s-1}\leq -d_{s-1}<0
	\end{multline*}
	and $1\leq q+p\leq n-1$, then Kodaira's vanishing theorem shows that \eqref{Cohom.irreg.l=0} is impossible. If \eqref{Cohom.Diff.l=0} holds, the second quadrant spectral sequence associated to the complex \eqref{Koszul resolution} twisting with $\Omega_M^{p-k}(-j_1 d_1-\dots-j_r d_r)$ implies that there exists an integer $0\leq j\leq r$ such that
	\[H^{q+k+j}(M,\Omega_M^{p-k}(-j_1-\dots-j_r)\otimes\wedge^j E)\not=0.\]
	As $q+j+p\leq n+r-1$ by assumption, then the Azizuki-Nakano vanishing theorem implies  
	\[-j_1-\dots-j_r-d_{ij}\geq 0\]
	for some $d_{ij}$. Since $d_{ij}>0$ if $j>0$ and $j_s\geq 0$ $(1\leq s\leq r)$, we obtain 
	\[j=j_1=\dots=j_r=0.\]
	Equivalently, we have $H^q(M,\Omega_M^p)\not=0$, and it follows that $p=q$.
	
	\textit{Proof of (2).} If $H^q(Y,\Omega_Y^p(\ell))\not=0$, by Lemma \ref{cohomology hypersurface to ambiant space} again, there exist integers $j_1,\dots,j_r\in\bbN$ with $0\leq k\leq p-1$ where $k=j_1+\dots+j_r$ such that we have either
	\begin{equation}\label{Cohom.Diff.l}
		H^{q+k}(Y,\Omega_M^{p-k}(\ell-j_1 d_1-\dots-j_r d_r)\vert_Y)\not=0
	\end{equation}
	or
	\begin{equation}\label{Cohom.irr.l}
		H^{q+p}(Y,\cO_Y(\ell-pd_{s-1}-(d_r-d_{s-1})j_r-\dots-(d_s-d_{s-1})j_s)\not=0
	\end{equation}
	for some $2\leq s\leq r+1$. 
	
	\textit{1st Case. \eqref{Cohom.Diff.l} holds.} In this case, the second quadrant spectral sequence associated to the complex \eqref{Koszul resolution} twisting with $\Omega_M^{p-k}(\ell-j_1 d_1-\dots-j_r d_r)$ implies that there exists an integer $0\leq j\leq r$ such that
	\[H^{q+k+j}(M,\Omega_M^{p-k}(\ell-j_1 d_1-\dots-j_r d_r)\otimes\wedge^j E)\not=0.\]
	As $q+j+p\leq n+r-1$, the Akizuki-Nakano vanishing theorem shows that we have
	\begin{equation}\label{AN.inequality}
	    \ell-j_1d_1-\dots-j_rd_r-d_{ij}\geq 0
	\end{equation}
	for all $i$. 
	
	If the equality in \eqref{AN.inequality} holds for all $d_{ij}$, by Theorem \ref{Vanishing complete intersection} (1), we get $q+k+j=p-k$. As a consequence,
	\begin{align*}
		n(\ell+q) & = n(j_1 d_1+\dots+j_r d_r+d_{ij}+p-2k-j)\\
		          & = n(p+(d_1-2)j_1+\dots+(d_r-2)j_r+d_{ij}-j)
	\end{align*}
	As $d_i\geq 2$, $d_{ij}\geq 2j$ and $r_Y\leq n-1$, we obtain $n(\ell+q)\geq np>r_Yp$. 
	
	If the inequality \eqref{AN.inequality} is strict for some $d_{ij}$, we can apply \cite[Theorem D]{BiswasChaputMourougane2015} to get
	\[\frac{\ell-j_1 d_1-\dots-j_r d_r-d_{ij}+q+k+j}{p-k}\geq \frac{r_M}{n+r}.\]
	As $d_{ij}\geq 2j$, $d_i\geq 2$, $r_M\leq n+r$, it follows that
	\begin{align*}
		\ell+q 
		&\geq p\frac{r_M}{n+r}+\left(d_1-1-\frac{r_M}{n+r}\right)j_1+\dots+\left(d_r-1-\frac{r_M}{n+r}\right)j_r+d_{ij}-j\\
		&\geq p\frac{r_M}{n+r}.
	\end{align*}
	Since $r_Y\leq n-2r$, it is easy to see $r_M/(n+r)>r_Y/n$, and we get the desired inequality.
	
	\textit{2nd Case. \eqref{Cohom.irr.l} holds.} As $1\leq q+p\leq n-1$, by Kodaira's vanishing theorem, we get
	\[\ell-pd_{s-1}-(d_r-d_{s-1})j_r-\dots-(d_s-d_{s-1})j_s\geq 0.\]
	As a consequence, we obtain
	\begin{align*}
		\ell+q & \geq pd_{s-1}+(d_r-d_{s-1})j_r+\dots+(d_s-d_{s-1})j_s+q\\
		       & \geq d_s j_s+\dots+d_r j_r+(p-k)d_{s-1}+q	  
	\end{align*}
	From the assumptions $p>k$ and $d_{s-1}\geq 2$, we get $\ell+q\geq 2p$. In particular, we have $n(\ell+q)\geq 2np>pr_Y$.
\end{proof}

\begin{remark}
	In view of our proof of Theorem \ref{Vanishing complete intersection} (1), we see that it holds even without the assumption $d_i\geq 2$. However, in the proof of Theorem \ref{Vanishing complete intersection} (2), the assumption $d_i\geq 2$ is necessary. 
\end{remark}

Theorem \ref{stability hypersurfaces of Hermitian symmetric spaces} is a direct consequence of Theorem \ref{Vanishing complete intersection}.

\begin{proof}[Proof of Theorem \ref{stability hypersurfaces of Hermitian symmetric spaces}]
	To prove the stability of $T_Y$, it is equivalent to prove the stability of $\Omega_Y^1$. Let $\cF\subset\Omega_Y^1$ be a nonzero proper subsheaf of rank $p$ ($1\leq p\leq n-1$). By Lefschetz hyperplane theorem and our assumption, we have $\pic(Y)\cong \bbZ\cO_Y(1)$, where $\cO_Y(1)=\cO_M(1)\vert_Y$. Thus, we could denote by $\ell$ the unique integer such that $\det(\cF)\cong\cO_Y(-\ell)$. Then we have $H^0(Y,\Omega_Y^p(\ell))\not=0$ by assumption. Since $p\leq n-1$, the Akizuki-Nakano vanishing theorem implies $\ell\geq 0$. As $p\geq 1$, Theorem \ref{Vanishing complete intersection} (1) implies $\ell>0$, and hence $\mu(\cF)<\mu(\Omega_Y^1)$ by Theorem \ref{Vanishing complete intersection} (2). In particular, it follows that $\Omega_Y^1$ is stable.
\end{proof}

\subsection{Effective restriction of tangent bundles}

In this subsection, we proceed to prove various effective restriction theorems for tangent bundles of complete intersections in irreducible Hermitian symmetric spaces of compact type. We use the standard cohomological arguments as in the proof of Theorem \ref{stability hypersurfaces of Hermitian symmetric spaces} to reduce the problem to the existence of twisted vector fields. 

\begin{prop}\label{reduce to exitence of vector fields}
	For $n\geq 3$ and $r\geq 1$, let $M$ be an $(n+r)$-dimensional irreducible Hermitian symmetric space of compact type. Let $Y=H_1\cap\dots\cap H_r$ be a smooth complete intersection of dimension $n$ with $H_i\in\vert\cO_M(d_i)\vert$ with $d_i\geq 2$. Let $X\in\vert \cO_Y(d)\vert$ be a smooth divisor. Assume that the composite of restrictions
	\[\pic(M)\rightarrow \pic(Y)\rightarrow \pic(X)\]
	is surjective. Moreover, if $Y$ is isomorphic to some smooth quadric hypersurface $Q^{n}$, we assume $d\geq 2$. Then the vector bundle $T_Y\vert_X$ is stable if and only if $H^0(X,T_Y(t)\vert_X)=0$ for any integer $t\leq -r_Y/n$, where $r_Y$ is the unique integer such that $\cO_Y(-K_Y)\cong\cO_Y(r_Y)$. 
\end{prop}

\begin{proof}
	The "only if" implication follows from the definition of stability. Now we assume $H^0(X,T_Y(t)\vert_X)=0$ for any $t\leq -r_Y/n$. Note that $T_Y\vert_X$ is stable if and only if $\Omega_Y^1\vert_X$ is stable. Let $\cF$ be a proper subsheaf of $\Omega_Y^1\vert_X$ of rank $p$. After replacing $\cF$ by its saturation in $\Omega_Y^1\vert_X$, we may assume that $\cF$ is saturated. We denote by $\ell$ the unique integer such that $\det(\cF)=\cO_X(-\ell)$. Then, by assumption, we get $H^0(X,\Omega_Y^p(\ell)\vert_X)\not=0$. To prove the stability of $\Omega_Y^1\vert_X$, it suffices to show that the following inequality
	\[\mu(\cF)=\frac{-\ell}{p}\cO_X(1)^{n-1}<\mu(\Omega_Y^1\vert_X)=\frac{-r_Y}{n}\cO_X(1)^{n-1}\]
	holds for all pairs of integers $(\ell,p)$ such that $H^0(X,\Omega_Y^p(\ell)\vert_X)\not=0$ and $1\leq p\leq n-1$. We consider the following exact sequence
	\[0\rightarrow\Omega_Y^p(\ell-d)\rightarrow\Omega_Y^p(\ell)\rightarrow \Omega_Y^p(\ell)\vert_X\rightarrow 0.\]
	As $1\leq p\leq n-1$ and $n\geq 3$, by Theorem \ref{Vanishing complete intersection}, $H^0(Y,\Omega_Y^p(\ell))\not=0$ implies $n\ell>pr_Y$. Thus we may assume $H^0(Y,\Omega_Y^p(\ell))=0$. Then the fact $H^0(X,\Omega_Y^p(\ell)\vert_X)\not=0$ implies $H^1(Y,\Omega_Y^p(\ell-d))\not=0$. 
	
	\textit{1st Case. $p\leq n-2$.} As $p\geq 1$, if $\ell\not=d$, by Theorem \ref{Vanishing complete intersection} (2), we have 
	\[n\ell\geq n(\ell-d+1)>pr_Y.\]
	If $\ell=d$, then we have $p=1$ by Theorem \ref{Vanishing complete intersection} (1). As a consequence, we get
	\begin{equation}\label{Vect.Fields.Ineq.}
		n\ell=nd>pr_Y=r_Y
	\end{equation}
	unless $d=1$ and $r_Y\geq n$. If $r_Y\geq n$, by Kobayashi-Ochiai theorem, $Y$ is isomorphic to either $\bbP^n$ or $Q^n$. As $d_i\geq 2$, then $Y$ is actually isomorphic to $Q^n$. However, by our assumption, if $Y\cong Q^n$, then we have $d\geq 2$. Hence the inequality \eqref{Vect.Fields.Ineq.} still holds in this case.
	
	\textit{2nd Case. $p=n-1$.} We denote by $\cQ$ the quotient $\left(\Omega_Y^1\vert_X\right)/\cF$. Since $\cF$ is saturated, the quotient $\cQ$ is a torsion-free coherent sheaf of rank one such that $\det(\cQ)=\cQ^{**}\cong\cO_X(-r_Y+\ell)$. Since $\cQ^{*}$ is a subsheaf of $T_Y\vert_X$, we get
	\[H^0(X,T_Y(\ell-r_Y)\vert_X)\not=0.\]
	By our assumption, we get $\ell-r_Y>-r_Y/n$. As a consequence, we get 
	\[\frac{\ell}{p}=\frac{\ell}{n-1}>\frac{r_Y}{n}.\]
	We have thus proved the proposition.
\end{proof}

As an application of Proposition \ref{reduce to exitence of vector fields}, we can derive Theorem \ref{Effective restriction HSS} by the nonexistence of global twisted vector fields.

\begin{proof}[Proof of Theorem \ref{Effective restriction HSS}]
	Let $X$ be a projective manifold of dimension $N\geq 2$, and let $L$ be an ample line bundle. Recall that $H^0(X,T_X\otimes L^{-1})\not=0$ if and only if $X\cong\bbP^N$ and $L\cong\cO_{\bbP^N}(1)$ (cf. \cite{Wahl1983}). In particular, if $M$ is not isomorphic to a projective space, then we have $H^0(M,T_M(t))=0$ for any integer $t<0$. 
	
	\textit{Proof of (1).} Under our assumption, by Theorem \ref{surjectivity twisted global vector fields}, the natural restriction map 
	\[\rho_t\colon H^0(M,T_M(t))\rightarrow H^0(X,T_M(t)\vert_X)\]
	is surjective for all $t\in\bbZ$. In particular, we have $H^0(X,T_M(t)\vert_X)=0$ for all $t<0$. This implies $H^0(X,T_Y(t)\vert_X)=0$ for all $t<0$ since $H^0(X,T_Y(t)\vert_X)$ is a subgroup of $H^0(X,T_M(t)\vert_X)$. As $Y$ is Fano, we have $r_Y>0$ and we conclude by Proposition \ref{reduce to exitence of vector fields}.
	
	\textit{Proof of (2).} By Theorem \ref{Surjective complete intersection Projective spaces}, the natural restriction map $\rho_t\colon H^0(Y_1,T_{Y_1}(t))\rightarrow H^0(X,T_{Y_1}(t)\vert_X)$ is surjective for all $t\leq -1$ if $d\geq d_1$. In particular, it follows that we have $H^0(X,T_{Y_1}(t)\vert_{X})=0$ for all $t\leq -1$ if $d\geq d_1$. Again, since $Y$ is Fano, we have $r_Y>0$ and the result is again a consequence of Proposition \ref{reduce to exitence of vector fields}.
	
	\textit{Proof of (3).} If $M$ is isomorphic to any smooth quadric hypersurface $Q^{n+r}$ and $Y\subset M$ is a complete intersection of degree $(d_1,\dots,d_r)$ such that $d_i\geq 2$ for all $1\leq i\leq r$. Then $Y$ is also a complete intersection in $\bbP^{n+r+1}$ of degree $(2,d_1,\dots,d_r)$ and the result follows immediately from (2).
	
	\textit{Proof of (4).} Since every smooth complete intersection in a quadric hypersurface is also a smooth complete intersection in some projective space, we may assume that $M$ is not isomorphic to any smooth quadric hypersurface. Note that we have $d>d_1-r_Y/n\geq 1$ as $d_1\geq 2$ and $r_Y\leq n$. Thus, by Proposition \ref{reduce to exitence of vector fields}, it suffices to show that 
	\[H^0(X,T_Y(t)\vert_X)=0\qquad \textrm{for}\qquad t\leq -\frac{r_Y}{n}.\] 
	Since $M$ is not isomorphic to any smooth quadric hypersurface, by Theorem \ref{surjectivity twisted global vector fields}, the natural restriction map
	\[H^0(M,T_M(t))\rightarrow H^0(Y,T_M(t)\vert_Y)\rightarrow H^0(X,T_M(t)\vert_X)\]
	is surjective for all $t\in\bbZ$. Let $\sigma\in H^0(X,T_Y(t)\vert_X)$ be a global section. Then $\sigma$ is also a global section of $T_M(t)\vert_X$. Therefore there exists a global twisted vector field $\widetilde{\sigma}\in H^0(M,T_M(t))$ such that $\widetilde{\sigma}\vert_X=\sigma$. Denote by $Y_j$ $(1\leq j\leq r)$ the scheme-theoretic complete intersection $H_1\cap\dots\cap H_j$. Then each $Y_j$ is smooth in a neighborhood of $Y$, and $\widetilde{\sigma}\vert_X=\sigma$ is a global section of $T_{Y_j}(t)\vert_X$ for all $1\leq j\leq r$. Consider the following exact sequence
	\[0\rightarrow T_{Y_j}(t)\vert_Y\rightarrow T_{Y_{j-1}}(t)\vert_Y\xrightarrow{\beta_j(t)} \cO_Y(d_j+t)\rightarrow 0.\]
	Then $\sigma\in H^0(X,T_Y(t)\vert_X)$ implies that the image $\widehat{\beta}_j(t)(\widetilde{\sigma}\vert_{Y})$ vanishes over $X$, where $\widehat{\beta}_j(t)$ is the induced map
	\[H^0(Y,T_{Y_{j-1}}(t)\vert_Y)\longrightarrow H^0(Y,\cO_Y(d_j+t)).\]
	However, since $X$ is general and $d>d_j+t$ for any $1\leq j\leq r$ by our assumption, we get $\widehat{\beta}_j(t)(\widetilde{\sigma}\vert_Y)=0$ for any $1\leq j\leq r$. Therefore we have $\widetilde{\sigma}\vert_Y\in H^0(Y,T_{Y_j}(t)\vert_Y)$ for any $1\leq j\leq r$, i.e., $\widetilde{\sigma}\vert_Y\in H^0(Y,T_Y(t))$. On the other hand, since $T_Y$ is stable (cf. Theorem \ref{stability hypersurfaces of Hermitian symmetric spaces}), we have 
	\[H^0(Y,T_Y(t))=0\qquad \textrm{for}\qquad t\leq -\frac{r_Y}{n}.\] 
	Hence, we obtain $\widetilde{\sigma}\vert_{Y}=0$ and consequently $\sigma=0$.
\end{proof}

Though the statements (2), (3) and (4) in the theorem are not optimal, they have the advantage to give a lower bound which is quite easy to compute. The next theorem deals with the case in which $Y$ is a general smooth hypersurface of $\bbP^{n+1}$. The proof of this theorem can be completed by combining Theorem \ref{Globalsections} and the method analogous to that used above.

\begin{proof}[Proof of Theorem \ref{Restrition Invariant Picard Group}]
	If $Y$ is isomorphic to either $\bbP^n$ or $Q^n$, it follows from \cite[Theorem A]{BiswasChaputMourougane2015}. So we shall assume that $Y$ is a general smooth hypersurface defined by a homogeneous polynomial $h$ of degree $d_h\geq 3$. By Proposition \ref{reduce to exitence of vector fields}, it is enough to prove that $H^0(X,T_Y(t)\vert_X)=0$ for $t\leq -r_Y/n$. As $n\geq 3$,  $d_h\geq 3$ and $r_Y=n+2-d_h$, we have
	\begin{align*}
	\left(\frac{\rho+d}{2}-d_h\right)-\left(-\frac{r_Y}{n}\right) & \geq \frac{(d_h-2)(n+2)+d}{2}-d_h-\left(\frac{d_h}{n}-\frac{n+2}{n}\right) \\
	& \geq \left(\frac{n}{2}-\frac{1}{n}\right)d_h-n-\frac{1}{2}+\frac{2}{n}\\
	&\geq \frac{3n}{2}-\frac{3}{n}-n-\frac{1}{2}+\frac{2}{n}\\
	&>0.
	\end{align*}
	This implies $-r_Y/n<(\rho+d)/2-d_h$. According to Theorem \ref{Globalsections}, the map
	\[H^0(Y,T_Y(t))\rightarrow H^0(X,T_Y(t)\vert_X)\]
	is surjective for $t\leq -r_Y/n$. The theorem is now a direct consequence of the stability of $T_Y$.
\end{proof}

\begin{remark}
	In Theorem \ref{Restrition Invariant Picard Group}, if $Y$ is an arbitrary smooth hypersurface, then the argument above does not work, since the strong Lefschetz property (SLP) of Milnor algebras of smooth hypersurfaces is still open.
\end{remark}

\section{Hyperplane of cubic threefolds}\label{Section Hyperplane of cubic threefolds}

In this section, we consider the case where the map $\pic(Y)\rightarrow\pic(X)$ is not surjective. By Noether-Lefschetz theorem mentioned in the introduction, this happens if $X$ is a quadric section of a quadric threefold $Q^3$, or $X$ is a quadric surface in $\bbP^3$, or $X$ is a cubic surface in $\bbP^3$. In these cases, $X$ is always a del Pezzo surface, i.e., the anti-canonical divisor $-K_X$ is ample. 

\subsection{Projective one forms}We denote by $\pi\colon S_r\rightarrow \bbP^2$ the surface obtained by blowing-up $\bbP^2$ at $r (\leq 8)$ points $p_1,\dots,p_r$ in general position and denote by $E_j$ the exceptional divisor over $p_j$. Then $S_r$ is a del Pezzo surface with degree $K_S^2=9-r$. It is well-known that the cotangent bundle $\Omega_{S_r}^1$ is stable with respect to the anti-canonical polarization $-K_{S_r}$ for $r\geq 2$ (cf. \cite{Fahlaoui1989}). There is a one-to-one correspondence between the saturated rank one subsheaves of $\Omega_{S_r}^1$ and the global sections of $\Omega_{\bbP^2}^1(a)$ which vanish only in codimension two. The global sections of $\Omega_{\bbP^2}^1(a)$ are usually called \textit{projective one forms}. In the following we give a brief description of this correspondence and we refer the reader to \cite{Fahlaoui1989} for further details. 

On one hand, let $L$ be a saturated rank one subsheaf of the cotangent bundle of $\Omega_{S_r}^1$. Then we have
\[c_1(L)=-aE_0-\sum_{j=1}^{9-r} b_j E_j\]
for some integers $a,b_j\in\bbZ$, where $E_0$ is the pull-back of a line in $\bbP^2$. The global section of $\Omega_{\bbP^2}^1(a)$ associated to $L$ is defined by the following map
\begin{multline*}
H^0(S_r,\Omega_{S_r}^1\otimes L^{-1}) \hookrightarrow H^0\left(S_r\setminus\cup E_i,\Omega_{S_r}^1\otimes L^{-1}\vert_{S_r\setminus\cup E_i}\right) \\
\cong H^0\left(\bbP^2\setminus\cup\{p_i\},\Omega_{\bbP^2}^1(a)\vert_{\bbP^2\setminus\cup\{p_i\}}\right)=H^0(\bbP^2,\Omega_{\bbP^2}^1(a)).
\end{multline*}

On the other hand, let $\omega\in H^0(\bbP^2,\Omega_{\bbP^2}^1(a))$ be a global section vanishing only in codimension two. Then $\omega$ can be identified with a rational global section of $\Omega_{\bbP^2}^1$ with pole supported on a line $T$. Let $E_0$ be the pull-back of $T$ by $\pi$. Then $\pi^*\omega$ is a global section of $\Omega_{S_r}^1\otimes\cO_{S_r}(aE_0)$. Let $\textrm{div}(\pi^*\omega)$ be the divisor defined by the zeros of $\pi^*\omega$. The saturated rank one subsheaf associated to $\omega$ is defined to the image of the induced morphism $\cO_{S_r}(-aE_0+\textrm{div}(\pi^*\omega))\rightarrow\Omega_{S_r}^1$. 

\begin{example}\label{Examples} We recall several examples given in \cite{Fahlaoui1989}.
	\begin{enumerate}[label=(\arabic*)]
		\item The form $\omega=x_0dx_1-x_1dx_0\in H^0(\bbP^2,\Omega^1_{\bbP^2}(2))$ defines a saturated subsheaf of $\Omega_{S_r}^1$ which is isomorphic to $\cO_{S_r}(-2E_0+2E_j)$, where $E_j$ is the exceptional divisor above $[0:0:1]$. Moreover, the only rank one saturated subsheaves of $\Omega_{S_r}^1$ with $a=2$ are $\cO_{S_r}(-2E_0+2E_j)$ and $\cO_{S_r}(-2E_0)$.
		
		\item We choose four points $p_1=[1:0:0]$, $p_2=[0:1:0]$, $p_3=[0:0:1]$ and $p_4=[1:1:1]$ in $\bbP^2$. Then the form defined by
		\[\omega=(x_1^2 x_2-x_2^2 x_1)dx_0+(x_2^2 x_0-x_0^2 x_2)dx_1+(x_0^2 x_1-x_1^2 x_0)dx_2\]
		induces a saturated subsheaf $L$ of $\Omega_{S_r}^1$ such that $c_1(L)=-4E_0+2\sum_{j=1}^4 E_j$, and there does not exist a subsheaf $L'$ of $\Omega_{S_r}^1$ such that 
		\[c_1(L')=-4E_0+2\sum_{j=1}^4 E_j+ E_i\]
		for $5\leq i\leq r$. In fact, let $\omega'$ be the corresponding projective one form of $L'$. Since $L$ is a subsheaf of $L'$, $\pi^*\omega'/\pi^*\omega$ can be identified to be a global section of $L'\otimes L^*$ whose zeros are supported on $E_i$. In particular, $\omega$ is actually proportional to $\omega'$. Nevertheless, the zeros of $\omega$ are $p_1,\dots,p_4$ and the points $[0:1:1]$, $[1:0:1]$ and $[1:1:0]$. Since there are at most four points of these points which are in general position, the exceptional line $E_i$ cannot be a zero of $\pi^*\omega'$.
	\end{enumerate}
\end{example}

In order to prove Theorem \ref{Semi-stability Cubic}, we need the following lemma due to Fahlaoui.

\begin{lemma}\cite[Lemme 1]{Fahlaoui1989}\label{Existence of section}
	Let $L$, $L'$ be two saturated subsheaves of $\Omega_{S_r}^1$. If $L$ is not isomorphic to $L'$, then we have $h^0(S_r,\cO_{S_r}(K_{S_r})\otimes L^{*}\otimes L'^{*})\geq 1$.
\end{lemma}

\subsection{Subsheaves of cotangent bundles of cubic surfaces} A cubic surface $S\subset\bbP^3$ is a blow-up $\pi\colon S\rightarrow\bbP^2$ of six points $p_j$ on $\bbP^2$ in general position. The exceptional divisor $\pi^{-1}(p_j)$ is denoted by $E_j$. Let $K_S$ be the canonical divisor of $S$ and $E_0$ the pull-back of a line in $\bbP^2$. Then we have 
\[-K_S=3E_0-\sum_{j=1}^6 E_j\sim H\vert_S,\]
where $H\in\vert\cO_{\bbP^3}(1)\vert$ is a hyperplane in $\bbP^3$. Let us recall the following well-known classical result of cubic surfaces.

\begin{itemize}
	\item There are exactly $27$ lines lying over a cubic surface: the exceptional divisors $E_j$ above the six blown up points $p_j$, the proper transforms of the fifteen lines in $\bbP^2$ which join two of the blown up points $p_j$, and the proper transforms of the six conics in $\bbP^2$ which contain all but one of the blown up points.
\end{itemize}

The following result gives an upper bound for the degree of the saturated subsheaves of $\Omega_S^1$. 

\begin{prop}\label{Invertible subsheaf}
	Let $S$ be a cubic surface and let $L\subset \Omega_S^1$ be a saturated invertible subsheaf. Then we have
	\[c_1(L)\cdot (-K_{S})\leq -3.\]
\end{prop}

\begin{proof}
	Note that since $\mu(\Omega_S^1)=-3/2$ and $\Omega_S^1$ is stable, we get $c_1(L)\cdot(-K_S)\leq -2$. Suppose that we have $c_1(L)\cdot K_S=2$ for some $L$. Then we have 
	\[c_1(L)=-aE_0-\sum_{j=1}^6 b_j E_j\]
	for some $a,b_j\in\bbZ$ with $a\geq 2$. If $a=2$, then $L$ is isomorphic to $\cO_S(-2E_0)$ or some $\cO_S(-2E_0+2E_i)$. In the former case we have $c_1(L)\cdot K_S=6$ and in the latter case we have $c_1(L)\cdot K_S=4$. So we need only consider the case $a\geq 3$. For fixed $i$, applying Lemma \ref{Existence of section} to $L$ and $-2E_0+2E_i$, we obtain an effective divisors $C_i$ such that
	\[C_i\sim K_S-L-(-2E_0+2E_i)=(a-1)E_0+(b_i-1)E_i+\sum_{j\not=i}(b_j+1)E_j.\]
	Denote by $d=-K_S\cdot C_i=3a+\sum_{j=1}^6 b_j+1$ the degree of $C_i$. By the assumption $c_1(L)\cdot K_S=2$, we obtain $3a+\sum_{j=1}^6 b_j=2$, and hence $d=3$. Moreover, as $a\geq 3$, we have $\sum_{j=1}^6 b_j\leq -7$, and consequently there is at least one $b_j\leq -2$. The proof will be divided into three steps.
	
	\textit{Step 1. We will show $b_j\geq -2$ for all $1\leq j\leq 6$.} There exist some $\pi$-exceptional effective divisors $\sum_{j=1}^6 c_{ij}E_j$ such that the effective divisors $C'_i$ defined as
	\[C'_i=C_i-\sum_{j=1}^6 c_{ij}E_j\sim(a-1)E_0+(b_i-c_{ii}-1)E_i+\sum_{j\not=i}(b_j-c_{ij}+1)E_j\]
	don't contain $\pi$-exceptional components. We denote the integer $b_j-c_{ij}$ by $b_{ij}$ and denote the degree $-K_S\cdot C'_i$ of $C'_i$ by $d'_i$, then we have 
	\begin{equation}\label{Sub b}
	b_{ij}\leq b_j\ \ \textrm{and}\ \  d'_i\leq d.
	\end{equation}
	Since the exceptional divisor $E_i$ is a line on $S$ and $-K_S\sim H\vert_X$ for some hyperplane $H\subset\bbP^3$, we have $\bs\vert-K_S-E_i\vert\subset E_i$. Moreover, since $C'_i$ does not contain $E_i$, we obtain
	\begin{equation}\label{Inequality bii}
	(-K_S-E_i)\cdot C'_i\geq 0\ \ \textrm{and}\ \ -b_{ii}+1=C'_i\cdot E_i\leq -K_S\cdot C'_i=d'_i.
	\end{equation}
	Combining \eqref{Sub b} and \eqref{Inequality bii} gives
	\begin{equation}\label{Inequality of b}
	-b_i\leq -b_{ii}\leq d'_i-1\leq d-1=2.
	\end{equation} 
	Since $i$ is arbitrary, we deduce that $b_i\geq -2$ for $i=1,\dots, 6$.
	
	\textit{Step 2. We show $b_j\leq-1$ for all $1\leq j\leq 6$ and $\sum_{j=1}^6 b_j\leq -8$.} Since there is at least one $b_j\leq -2$ and $b_i\geq -2$ for all $i$, without loss of generality we may assume $b_1=-2$. As a consequence of inequality \eqref{Inequality of b}, we have 
	\[b_{11}=-2\ \ \textrm{and}\ \ d'_1=d=3.\] 
	It follows that $-K_S\cdot(C_i-C_i')=d-d'_1=0$. As $C_i-C_i'\geq 0$ and $-K_S$ is ample, we obtain $C'_1=C_1$ and
	\begin{equation}\label{Equality of Degrees}
	-K_S\cdot C_1=E_1\cdot C_1=3.
	\end{equation}	
	Since $C_1$ does not contain $E_j$, we must have $-b_j-1=C_1\cdot E_j\geq 0$ for $j\geq 2$, which yields $b_j\leq -1$ for $j\geq 2$. As a consequence, we get
	\[-12\leq\sum_{j=1}^6 b_j\leq -7\ \ \textrm{and}\ \ 3\leq a\leq 4.\]
	
	Let $C_{1\ell}$ be a component of $C_1$. Since $\Bs\vert-K_S-E_1\vert\subset E_1$ and $C_1$ does not contain $E_1$, we have $(-K_S-E_1)\cdot C_{1\ell}\geq 0$. Then the equality (\ref{Equality of Degrees}) implies $(-K_S-E_1)\cdot C_{1\ell}=0$. Therefore, $C_{1\ell}$ is actually a plane curve and there exists a plane $H_{\ell}\subset\bbP^3$ such that $C_{1\ell}+E_1\leq H_{\ell}\vert_S$. In particular, we have 
	\[-K_S\cdot C_{1\ell}=H_\ell\vert_S\cdot C_{1\ell}\leq 2.\]
	On the other hand, as $-K_S\cdot C_1=3$, it follows easily that there exists at least one component of $C_1$, denoted by $C_{11}$, such that $-K_S\cdot C_{11}=1$. In particular, $C_{11}$ is a line on $S$. However, $C_{11}$ is not $\pi$-exceptional, so the line $C_{11}$ passes at least two $\pi$-exceptional divisors, and it follows that there exists some $j\ (\geq 2)$ such that 
	\[-2\leq b_j=-1-C_1\cdot E_j\leq-1-C_{11}\cdot E_j\leq -2.\]
	Hence we obtain $\sum_{j=1}^6 b_j\leq -8$. 
	
	\textit{Step 3. We exclude the case $c_1(L)\cdot K_S=2$.} By our argument above, if $c_1(L)\cdot K_S=2$, then we have 
	\[a\geq 3,\ \ -2\leq b_j\leq -1\ \ \textrm{and}\ \ -12\leq \sum_{j=1}^6 b_j\leq -8.\] Then the equality $3a+\sum_{j=1}^6 b_j=2$ shows $a=4$ and $\sum_{j=1}^6 b_j=-10$, and consequently $L$ is a line bundle of the form
	\[-4E_0+2E_1+2E_2+2E_3+2E_4+E_5+E_6.\]
	Nevertheless, we have seen that such a line bundle cannot be a saturated subsheaf of $\Omega_{S}^1$ (cf.~Example \ref{Examples}), a contradiction.
\end{proof}

\subsection{Stability of restrictions of tangent bundles of general cubic threefolds} In this subsection, we will prove Theorem \ref{Semi-stability Cubic}. First we consider the saturated subsheaves of $\Omega^1_Y\vert_X$ of rank two and we give an upper bound for the degree of $c_1(\cF)$ with respect to $-K_X$.

\begin{lemma}\label{Rank two subsheaf}
	Let $Y$ be a general smooth cubic threefold and let $X\in \vert\cO_Y(1)\vert$ be a general smooth divisor. If $\cF\subset\Omega^1_Y\vert_X$ is a saturated subsheaf of rank two, then we have 
	\[c_1(\cF)\cdot (-K_X)\leq-5.\]
\end{lemma}

\begin{proof}
	The natural inclusion $\cF\subset\Omega_Y^1\vert_X$ implies $h^0(X,\Omega_Y^2\vert_X\otimes\det(\cF)^*)\geq 1$. Using the short exact sequence
	\[0\rightarrow \Omega_X^1(-1)\otimes \det(\cF)^{*}\rightarrow\Omega_Y^2\vert_X\otimes\det(\cF)^{*}\rightarrow\cO_X(K_X)\otimes\det(\cF)^{*}\rightarrow 0,\]
	we obtain either $h^0(X,\Omega_X^1(-1)\otimes\det(\cF)^{*})\geq 1$ or $h^0(X,\cO_X(K_X)\otimes \det(\cF)^{*})\geq 1$. In the former case, the stability of $\Omega_X^1$ implies
	\[(c_1(\cF)+c_1(\cO_X(1)))\cdot (-K_X)<\frac{K_X\cdot(-K_X)}{2}=-\frac{3}{2}.\]
	This yields
	\[c_1(\cF)\cdot(-K_X)<-c_1(\cO_X(1))\cdot (-K_X)-\frac{3}{2}=-\frac{9}{2}<-4.\]
	In the latter case, we have $c_1(\cF)\cdot (-K_X)\leq K_X\cdot (-K_X)=-3$ with equality if and only if $c_1(\cF)=-K_X$, and the quotient $\cG\colon=\left(\left.\Omega_Y^1\right\vert_X\right)/\cF$ is a torsion-free sheaf of rank one. 
	
	If $c_1(\cF)\cdot(-K_X)=-3$, then $\det(\cF)\cong\cO_X(K_X)\cong\cO_X(-1)$ and consequently $\det(\cG)=\cO_X(-1)$. Since $\cG^{*}$ is a subsheaf of $T_Y\vert_X$, we obtain 
	\[h^0(X,T_Y\vert_X\otimes \det(\cG))=h^0(X,T_Y(-1)\vert_X)\geq 1.\]
	Since $T_Y(-1)\vert_X$ is a subsheaf of $T_Y\vert_X$, we get $H^0(X,T_Y\vert_X)\not=0$. Then, by Theorem \ref{Globalsections}, we obtain $H^0(Y,T_Y)\not=0$. Nevertheless, it is well-known that there are no global holomorphic vector fields over a cubic threefold (cf.~\cite[Theorem 11.5.2]{KatzSarnak1999}). This leads to a contradiction.
	
	If $c_1(\cF)\cdot(-K_X)=-4$, then $\det(\cF)\cong\cO_X(-1)\otimes\cO_X(-T)$ for some line $T\subset X$. As a consequence, we have $\det(\cG)=\cO_X(-1)\otimes\cO_X(T)$. Since $\cG^{*}$ is a subsheaf of $T_Y\vert_X$, we get 
	\[h^0(X,T_Y(-C)\vert_X)>0,\]
	where $C$ is a conic such that $\cO_X(C)\cong \cO_X(1)\otimes\cO_X(-T)$. Note that the sheaf $T_Y(-C)\vert_X$ is a subsheaf of $T_Y\vert_X$, it follows that $H^0(X,T_Y\vert_X)\not=0$. Similarly, Theorem \ref{Globalsections} implies $H^0(Y,T_Y)\not=0$, which is impossible.
\end{proof}

Now we are in the position to prove the main theorem in this section.

\begin{proof}[Proof of Theorem \ref{Semi-stability Cubic}]
	It is enough to prove that $\Omega_Y^1\vert_X$ is stable with respect to $\cO_X(1)$. Since $\mu(\Omega_Y^1\vert_X)=-2$, it suffices to prove that the following inequality holds for any proper saturated subsheaf $\cF$ of $\Omega_Y^1\vert_X$.
	\[\mu(\cF)=\frac{c_1(\cF)\cdot -K_X}{\rk(\cF)}< -2\]

	\textit{1st Case. $\cF\subset \Omega_Y^1\vert_X$ is a saturated subsheaf of rank one.} Since $\cF$ is a reflexive sheaf of rank one and $X$ is smooth, $\cF$ is actually an invertible sheaf. Then the exact sequence 
	\[0\rightarrow\cO_X(-1)\otimes \cF^{*}\rightarrow \Omega_Y^1\vert_X\otimes \cF^{*}\rightarrow \Omega_X^1\otimes \cF^{*}\rightarrow 0\]
	implies that we have either $h^0(X,\cO_X(-1)\otimes \cF^{*})\geq 1$ or $h^0(X,\Omega^1_X\otimes \cF^{*})\geq 1$. In the former case, we have 
	\[\mu(\cF)=c_1(\cF)\cdot (-K_X)\leq c_1(\cO_X(-1))\cdot (-K_X)=-3<-2.\]
	In the latter case, let $\overline\cF$ be the saturation of $\cF$ in $\Omega_X^1$, then Proposition~\ref{Invertible subsheaf} implies 
	\[\mu(\cF)\leq \mu(\overline\cF)=c_1(\overline\cF)\cdot (-K_X)\leq-3.\]
	
	\textit{2nd Case. $\cF\subset\Omega_Y^1\vert_X$ is a saturated subsheaf of rank two.} In this case, by Lemma \ref{Rank two subsheaf}, we have
	\[\mu(\cF)=\frac{c_1(\cF)\cdot(-K_X)}{2}\leq \frac{-5}{2}<-2.\]
	This finishes the proof.
\end{proof}

\def\cprime{$'$} 

\renewcommand\refname{Reference}
\bibliographystyle{alpha}
\bibliography{stability}

\end{document}